\documentclass[a4paper,12pt]{article}

\usepackage[normalem]{ ulem }
\usepackage{soul}
\usepackage[utf8]{inputenc}
\usepackage[T1]{fontenc}
\usepackage{fixltx2e}
\usepackage{graphicx}
\usepackage{longtable}
\usepackage{float}
\usepackage{wrapfig}
\usepackage{rotating}
\usepackage[normalem]{ulem}
\usepackage{amsmath}
\usepackage{amsthm}
\usepackage{textcomp}
\usepackage{marvosym}
\usepackage{wasysym}
\usepackage{amssymb}
\usepackage{capt-of}
\usepackage{hyperref}
\usepackage{comment}
\usepackage{color}
\usepackage{subfigure}

\def\sizesmallfig{0.34}
\def\sizeellfig{0.43}
\def\sizephytofig{0.25}

\newcommand{\didier}[1]{{\color{green}\bf Didier: #1}}
\newcommand{\new}[1]{\textcolor{black}{#1}}
\def\Sb{\mathbf{S}}
\def\A{\mathbf{A}}
\def\b{\mathbf{b}}
\def\B{\mathbf{B}}
\def\G{\mathbf{G}}
\def\V{\mathbf{V}}
\def\X{\mathbf{X}}
\def\Y{\mathbf{Y}}
\def\x{\mathbf{x}}
\def\y{\mathbf{y}}
\def\p{\mathbf{p}}
\def\z{\mathbf{z}}
\def\M{\mathbf{M}}
\def\Q{\mathbf{Q}}
\newcommand{\R}{\mathbb{R}}
\newcommand{\C}{\mathbb{C}}
\newcommand{\N}{\mathbb{N}}

\theoremstyle{plain}
\newtheorem{theorem}{Theorem}[section]
\newtheorem{lemma}[theorem]{Lemma}

\theoremstyle{definition}

\newtheorem{assumption}[theorem]{Assumption}
\newtheorem{example}{Example}
\newtheorem{remark}{Remark}

\DeclareMathOperator{\vol}{vol}

\DeclareMathOperator{\trace}{trace}

\begin{document}

\author{Victor Magron$^{1}$ \and Pierre-Loic Garoche$^{2}$ \and Didier Henrion$^{1,3}$ \and Xavier Thirioux$^{4}$}
\date{\today}

\title{Semidefinite Approximations of Reachable Sets for Discrete-time Polynomial Systems}

\footnotetext[1]{CNRS; LAAS; 7 avenue du colonel Roche, F-31400 Toulouse; France}
\footnotetext[2]{ONERA; the French Aerospace Lab; France}
\footnotetext[3]{Universit\'e de Toulouse;  LAAS, F-31400 Toulouse; France}
\footnotetext[4]{Institut de Recherche en Informatique de Toulouse (IRIT)} 
\maketitle

\begin{abstract}

We consider the problem of approximating the reachable set of a discrete-time polynomial system from a semialgebraic set of initial conditions under general semialgebraic set constraints. 
Assuming inclusion in a given simple set like a box or an ellipsoid, we provide a method to compute certified outer approximations of the reachable set. 

The proposed method consists of building a hierarchy of relaxations for an infinite-dimensional moment problem. 
Under certain assumptions, the optimal value of this problem is the volume of the reachable set and the optimum solution is the restriction of the Lebesgue measure on this set. 
Then, one can outer approximate the reachable set as closely as desired with a hierarchy of super level sets of increasing degree polynomials. 
For each fixed degree, finding the coefficients of the polynomial 
boils down to computing the optimal solution of a convex semidefinite program. 
When the degree of the polynomial approximation tends to infinity, we provide strong convergence guarantees of the super level sets to the reachable set.
We also present some application examples together with numerical results.
\end{abstract}


\section{Introduction}
\label{sec:intro}

Given a dynamical polynomial system described by a discrete-time (difference) equation, 
the (forward) reachable set (RS) is the set of all states that can be reached from a set of initial conditions under general state constraints. This set appears in different fields such as optimal control, hybrid systems or program analysis.  
In general, computing or even approximating the RS is a challenge. Note that the RS is typically non-convex and non-connected, even in the case when the set of initial conditions is convex and the dynamics are linear.

{Computing or approximating RS has been a topic of intensive research in the last four decades. When the dynamics of the discrete-time system is linear, one can rely on contractive algorithms based on finite LP relaxations combined with polyhedral projections~\cite{Bertsekas72}. For more details and historical surveys, we refer the interested reader to~\cite{Blanchini99,Blanchini2007set} as well as~\cite{Donze10} for an extension to hybrid systems. A recent approach~\cite{Harwood2016} has extended the scope of problems for which one can construct polyhedral bounds on the reachable set. Again, the method relies on (parametric) LP and numerical integration procedures. The works~\cite{BenSassi2015} compares approaches combining LP relaxations with Bernstein decompositions or Krivine/Handelman representations of nonnegative polynomials (also based on sums of squares certificates). See also~\cite{BenSassi2012} using template polyhedra and Bernstein form of polynomials.
In recent work~\cite{Dreossi17}, the author relies on Bernstein expansions to over approximate the RS after in  discrete-time after a finite number of time steps. The same technique allows to perform parameter synthesis. 
}
However, all these methods based on linear relaxations often fail to construct tight bounds since they build convex over approximations of possibly nonconvex sets. Furthermore, they usually do not provide convergence guarantees.


{
Another classical approach relies on Lyapunov theory (see e.g.~\cite[\S~5.7]{Sontag90mathematicalcontrol}) in order to  approximate from outside. 
This can be done in a continuous-time setting (with possible extension to discrete-time systems), i.e.,~when the state variable is constrained from an initial condition to satisfy an ordinary differential equation $\dot{\x} = f(\x,t)$. 
The idea is to search for a Lyapunov function $v$ (also called value or Bellman function in the context of optimal control) which is negative on the set of initial conditions and with negative derivative of states satisfying some general constraints. 
These inequalities provide sufficient conditions for the RS to be included in the sublevel set of $v$. 
In the case where the set of initial (resp.~general) state constraints are defined by polynomial inequalities, the difficulty of computing such a funcion $v$ can be practically addressed. This is done while reducing the search space to polynomials of bounded degree and by replacing the inequalities satisfied by $v$ (and its total derivative) by stronger equality constraints involving $v$  and (weighted) sums of squares (SOS) of polynomials. Since the weights are the polynomials defining the set of initial and general constraints, computing $v$ together with these SOS polynomials boils down to solving a semidefinite program of fixed size. This general framework has been used in~\cite{Prajna04} for the safety verification of hybrid systems. In this case, the function $v$ is called a ``barrier certificate'' and can be constructed by computing an SOS decomposition. The zero level set of $v$ separates a given unsafe region from all possible trajectories starting from a prescribed set of initial conditions.  \newline
These dual Lyapunov certificates relying on SOS decompositions also allow to obtain approximations of the (backward) reachable set (also called region of attraction)~\cite{Chesi}. 
\new{In~\cite{Vannelli85}, the authors proved the existence of a Lyapunov function, whose sublevel set is the region of attraction of a given equilibrium point of a continuous-time system.} 
When the degree of the approximation $v$ is fixed in advance, one can obtain convergence guarantees by increasing the degree of the SOS polynomials. However, one has no guarantee that when the degree of $v$ goes to infinity, the approximation conservatism  asymptotically vanishes. In addition, the conservatism of such approximations relying on dual Lyapunov certificates is not easy to estimate in a systematic way.
}

In this paper, we propose a characterization of the RS as the solution of an infinite-dimensional linear programming (LP) problem.  This characterization is done by considering a hierarchy of converging convex programs through moment relaxations of the LP. Doing so, one can compute tight outer approximations of the RS. Such outer approximations yield invariants for the discrete-time system, which are sets where systems trajectories are confined. 

This general methodology is deeply inspired from previous research efforts. 
The idea of formulation relying on LP optimization over probability measures appears in~\cite{Lasserre01moments}, with a hierarchy of semidefinite programs (SDP) also called moment-sum-of-squares or sometimes Lasserre hierarchy, whose optimal values converge from below to the infimum of a multivariate polynomial. One can see outer  approximations of sets as the analogue of lower approximations of real-valued functions. 
In~\cite{HLS09vol}, the authors leverage on these techniques to address the problem of computing outer approximations by single polynomial super level sets of basic compact semialgebraic sets described by the intersection of a finite number of given polynomial super level sets. Further work focused on approximating semialgebraic sets for which such a description is not explicitly known or difficult to compute: in~\cite{Las13}, the author derives converging outer approximations of sets defined with existential quantifiers; in~\cite{MHL15image}, the authors approximate the image of a compact semialgebraic set under a polynomial map. The current study can be seen as an extension of~\cite{MHL15image} where instead of considering only one iteration of the map, we consider infinitely many iterations starting from a set of initial conditions.

This methodology has also been successfully applied for several problems arising in the context of polynomial systems control. Similar convergent hierarchies appear in~\cite{HK14roa}, where the authors approximate the region of attraction (ROA) of a controlled polynomial system subject to compact semialgebraic constraints in continuous time. {This framework is extended to hybrid systems in~\cite{Shia14}}. Note that the ROA is not a semialgebraic set in general. The authors of~\cite{KHJ13mci} build upon the infinite-dimensional LP formulation of the ROA problem while providing a similar framework to characterize the maximum controlled invariant (MCI) for discrete and continuous time polynomial dynamical systems. The framework used for ROA and MCI computation both rely on occupation measures. These allow to measure the time spent by solutions of differential or difference equations. As solutions of a linear transport equation called the Liouville Equation, occupation measures also capture the evolution of the semialgebraic set describing the initial conditions. 
As mentioned in~\cite{HK14roa}, the problem of characterizing the (forward) RS in a continuous setting and finite horizon could be done as ROA computation by using a time-reversal argument. 
{
The modeling power of this approach also extends to the analysis of attractors of dynamical systems, e.g.,~by approximating the moments and support of invariant measures in both continuous and discrete-time settings~\cite{KordaInv18,invsdp18}.
}
In the present study, we handle the problem in a discrete setting and infinite horizon. Our contribution follows a similar approach but requires to describe the solution set of another Liouville Equation. 


In contrast with previous work, our contributions are the following:
\begin{itemize}
\item we rely on a infinite-dimensional LP formulation to handle the general discrete-time RS problem under semialgebraic state and initial conditions;
\item we build a hierarchy of finite-dimensional SDP relaxations for this infinite-dimensional LP problem. 
Under additional assumptions, the optimal value of this LP is the volume of the RS, whose optimum is the restriction of the Lebesgue measure on the RS;
\item we use the solutions of these SDP relaxations to approximate the RS as closely as desired, with a sequence of certified outer approximations defined as (possibly nonconvex) polynomial super level sets, which is less restrictive than linear or convex approximations used in the literature (such as polyhedra or ellipsoids).
\end{itemize}

In the sequel, we focus on computation of semidefinite approximations of the forward RS for discrete-time polynomial systems. The problem statement is formalized in Section~\ref{sec:pb} and reformulated in Section~\ref{sec:infLP} into a primal optimization problem over probability measures satisfying Liouville's Equation. We explain how to obtain the dual problem in Section~\ref{sec:dualLP}.
Then, we show in Section~\ref{sec:sdp} how to solve in practice the primal problem with moment relaxations, as well as the dual with sums of squares strenghtenings. In both cases, this boils down to solving a hierarchy of finite-dimensional SDP problems. We illustrate the method with several numerical experiments in Section~\ref{sec:bench}.

\section{Problem Statement and Prerequisite}
\label{sec:pb}

\subsection{Forward Reachable set for Discrete-time Polynomial Systems}
\label{sec:reachset}
Given $r,n \in \N$, let $\R[\x]$ (resp.~$\R_{2r}[\x]$) stands for the vector space of real-valued $n$-variate polynomials (resp. of degree at most $2r$) in the variable $\x=(x_1,\ldots,x_n) \in \R^n$. 
We are interested in the polynomial discrete-time system defined by

\begin{itemize}
\item a set of initial constraints assumed to be compact basic semi-algebraic:
\begin{equation}
\label{eq:defX0}
 \X^0 :=  \{\x \in \R^n : g_1^0(\x)  \geq 0, \dots, g_{m^0}^0(\x) \geq 0 \}
\end{equation} 
defined by given polynomials $g_1^0,\ldots,g_{m^0}^0 \in \R[\x]$, $m^0 \in \N$;
\item a polynomial transition map $f : \R^n \to \R^n$, $\x \mapsto f(\x) := (f_1(\x), \dots, f_n(\x)) \in \R^n[\x]$ 
of degree $d:=\max \{\deg f_1,\ldots,\deg f_n\}$.
\end{itemize}
Given $T \in \N$, let us define the set of all admissible trajectories after at most $T$ iterations of the polynomial transition map $f$, starting from any initial condition in $\X^0$:
\[
\X^{T} := \X^0 \cup f(\X^0) \cup  f(f(\X^0)) \cup\dots \cup f^T(\X^0) \,,
\]
with $f^T$ denoting the $T$-fold composition of $f$.
Then, we consider the reachable set (RS) of all admissible trajectories:
\[
\X^{\infty} := \lim_{T \to \infty} \X^{T}
\]
and we make the following assumption in the sequel:
\begin{assumption}
\label{hyp:constraints}
The RS $\X^{\infty}$ is included in a given compact basic semi-algebraic set
\begin{equation}
\label{eq:defX}
 \X :=  \{\x \in \R^n : g_1(\x)  \geq 0, \dots, g_m(\x) \geq 0 \}
\end{equation} 
defined by polynomials $g_1,\ldots,g_m \in \R[\x]$, $m \in \N$.
\end{assumption}

\begin{example}
{Let us consider $\X^0=[1/2,1]$ and $f(x)=x/4$. Then $\X^{\infty} = [1/2,1]\cup[1/8,1/4]\cup[1/32,1/16]\ldots$ is included in the basic compact semialgebraic set $\X=[0,1]$, so  Assumption~\ref{hyp:constraints} holds. Note that $\X^{\infty}$ is not connected within $\X$.}
\end{example}

We denote the closure of $\X^{\infty}$ by $\bar{\X}^{\infty}$ . Obviously $\X^{\infty} \subseteq \bar{\X}^{\infty}$ and the inclusion can be strict. To circumvent this difficulty {later on in Section~\ref{sec:infLP} and Section~\ref{sec:sdp}}, we make the following assumption in the remainder of the paper.
\begin{assumption}
\label{hyp:closure}
The volume of the RS is equal to the volume of  its closure, i.e.~$\vol \X^{\infty} = \vol \bar{\X}^{\infty}$.
\end{assumption}

\begin{example}
\label{ex:half}
Let $\X^0=[1/2,1]$ and $f(x)=x/2$. Then $\X^{\infty} = [1/2,1]\cup[1/4,1/2]\cup[1/8,1/4]\ldots = (0,1]$ is a half-closed interval within $\X=[0,1]$. Note that $\bar{\X}^{\infty}=\X$,  so that {$\vol \X^{\infty} = \vol \bar{\X}^{\infty} = 1$ and} Assumption \ref{hyp:closure} is satisfied.
\end{example}

\subsection{Prerequisite and Working Assumptions}
\label{sec:pre}

Given a compact set $\A \subset \R^n$, we denote by $\mathcal{M}(\A)$ the vector space of finite signed Borel measures supported on $\A$, namely real-valued functions from the Borel sigma algebra $\mathcal{B}(\A)$. 
The support of a measure $\mu \in \mathcal{M}(\A)$ is defined as the closure of the set of all points $\x$ such that $\mu(\B) \neq 0$ for any open neighborhood $\B$ of $\x$.
We note $\mathcal{C}(\A)$ the Banach space of continuous functions on $\A$ equipped with the sup-norm.
Let $\mathcal{C}(\A)'$ stand for the topological dual of $\mathcal{C}(\A)$ (equipped with the sup-norm), i.e.~the set of continuous linear functionals of $\mathcal{C}(\A)$.
By a Riesz identification theorem {(see for instance~\cite{lieb2001analysis})}, $\mathcal{C}(\A)'$ is isomorphically identified with
$\mathcal{M}(\A)$ equipped with the total variation norm denoted by $\|
\cdot\|_{\text{TV}}$.
Let $\mathcal{C}_+(\A)$ (resp.~$\mathcal{M}_+(\A)$) stand for the cone of non-negative elements of $\mathcal{C}(\A)$ (resp. $\mathcal{M}(\A)$). 
The topology in $\mathcal{C}_+(\A)$ is the strong topology of uniform convergence in contrast with the weak-star topology in $\mathcal{M}_+(\A)$. See~\cite[Section 21.7]{Royden} and \cite[Chapter IV]{alexander2002course} or~\cite[Section 5.10]{Luenberger97} for functional analysis, measure theory and applications in convex optimization.

With $\X$ a basic compact semialgebraic set as in~\eqref{eq:defX}, the restriction of the Lebesgue measure on a subset $\A \subseteq \X$ is
$\lambda_\A (d \x) := \mathbf{1}_\A(\x) \, d \x $, 
where $\mathbf{1}_\A : \X \to \{0, 1\}$ stands for the indicator function of $\A$, namely $\mathbf{1}_\A(\x) = 1$ if $\x \in \A$ and
 $\mathbf{1}_\A(\x) = 0$ otherwise.

The moments of the Lebesgue measure on $\A$ are denoted by
\begin{equation}\label{momb}
y^\A_{\beta} := \int \x^{\beta} \lambda_\A(d\x) \in \R \,, \quad \beta \in \N^n
\end{equation}
where we use the multinomial notation $\x^{\beta} := x^{\beta_1}_1 x^{\beta_2}_2 \ldots x^{\beta_n}_n$. 
The Lebesgue volume of $\A$ is $\vol \A := y^\A_0 = \int \lambda_\A (d\x)$.

Given $\mu,\nu \in \mathcal{M}(\A)$,
the notation
\[
\mu \leq \nu
\]
stands for $\nu-\mu \in \mathcal{M}_+(\A)$, and we say that $\mu$ is {\em dominated} by $\nu$.

Given $\mu \in \mathcal{M}_+(\X)$, the so-called pushforward measure or {\em image measure}, see e.g.~\cite[Section 1.5]{AFP00}, of $\mu$ under $f$ is defined as follows:
\[
f_\# \mu (\A) := \mu (f^{-1}(\A)) = \mu (\{ \x \in \X : f(\x) \in \A \})
\]
for every set $\A \in \mathcal{B}(\X)$.
{The main property of the pushforward measure is the change-of-variable formula: $\int_\A v(\x) f_\# \mu (d \x) = \int_{f^{-1}(\A)} v(f(\x)) \mu (d \x)$, for all $v \in \mathcal{C}(\A)$.}
%


With $\X^0$ a basic compact semialgebraic set as in~\eqref{eq:defX0}, we set $r_j^0 := \lceil (\deg g_j^0 ) / 2 \rceil, j = 1, \dots, m^0$ and with $\X$ a basic compact semialgebraic set as in~\eqref{eq:defX}, we set $r_j := \lceil (\deg g_j ) / 2 \rceil, j = 1, \dots, m$.
Let $\Sigma[\x]$ stand for the cone of polynomial sums of squares (SOS) and let $\Sigma_r[\x]$ denote the cone of SOS polynomials of degree at most $2 r$, namely $\Sigma_r[\x] := \Sigma[\x] \cap \R_{2r}[\x]$.
 
For the ease of further notation, we set $g_0^0(\x) := 1$ and $g_0(\x) := 1$.
For each integer $r$, let $\Q_r^0$ (resp.~$\Q_r$) be the $r$-truncated quadratic module generated by $g_0^0, \dots, g_m^{m^0}$ (resp. $g_0, \dots, g_m$):
\begin{align*}
\Q_r^0 & := \Bigl\{\,\sum_{j=0}^{m^0} s_j(\x) {g_j^0} (\x) :  s_j \in \Sigma_{r - r_j^0}[\x], \,  j = 0, \dots, m^0 \, \Bigr\} \,,\\
\Q_r & := \Bigl\{\,\sum_{j=0}^{m} s_j(\x) {g_j} (\x) : s_j \in \Sigma_{r - r_j}[\x], \,  j = 0, \dots, m  \,\Bigr\} \,.
\end{align*}

To guarantee the convergence behavior of the relaxations presented in the sequel, we need to ensure that polynomials which are positive on $\X^0$ (resp.~$\X$) lie in $\Q_r^0$ (resp.~$\Q_r$) for some $r \in \N$. The existence of such SOS-based representations is guaranteed by Putinar's Positivstellensaz (see e.g.~\cite[Section 2.5]{lasserre2009moments}), when the following condition holds:
\begin{assumption}
\label{hyp:archimedean}
There exists a large enough integer $N^0$ (resp.~$N$) such that one of the polynomials describing the set $\X^0$ (resp.~$\X$) is equal to $g_0^i := N^0 - \| \x \|_2^2$ (resp.~$g^i := N - \| \x \|_2^2$).
\end{assumption}
{This assumption is slightly stronger than compactness. Indeed, compactness of $\X^0$ (resp.~$\X$) already ensures that each variable has finite lower and upper bounds. One (easy) way to ensure that Assumption~\ref{hyp:archimedean} holds is to add a redundant constraint involving a well-chosen $N^0$ (resp.~$N$) depending on these bounds, in the definition of $\X^0$ (resp.~$\X$).}

From now on, the over approximation set $\X$ of the set $\X^\infty$ is assumed to be ``simple'' (e.g.~a ball or a box), meaning that $\X$ fulfills the following condition: 
\begin{assumption}
\label{hyp:momb}
The moments~\eqref{momb} of the Lebesgue measure on $\X$ are available analytically.
\end{assumption}
{
\begin{remark}
Since we are interested in characterizing the RS of polynomial systems with bounded trajectories, Assumption~\ref{hyp:constraints} and Assumption~\ref{hyp:momb} are not restrictive. As mentioned above, Assumption~\ref{hyp:archimedean} can be ensured by using Assumption~\ref{hyp:constraints}. While relying on Assumption~\ref{hyp:closure}, we restrict ourselves to discrete-time systems where the boundary of the RS has zero Lebesgue volume.
\end{remark}
}

Let $\N_0$ stands for the set of positive integers.
For all $r \in \N$, we set $\N^{n}_r := \{ \beta \in \N^{n} : \sum_{j=1}^{n} \beta_j \leq r \}$, whose cardinality is $\binom{n+r}{r}$.
Then a polynomial $p \in \R[\x]$ is written as follows:
\[\x \mapsto p(\x) \,=\,\sum_{\beta \in\N^n} \, p_{\beta} \, \x^\beta \:, \]
and $p$ is identified with its vector of coefficients $\p=(p_{\beta})$ in the canonical basis $(\x^\beta)$, $\beta \in\N^n$.

Given a real sequence $\y =(y_{\beta})_{\beta \in \N^n}$, let us define the linear functional $\ell_\y : \R[\x] \to \R$ by $\ell_\y(p) := \sum_{\beta} p_{\beta} y_{\beta}$, for every polynomial $p$.

Then, we associate to $\y$ the so-called {\it moment matrix} $\M_r(\y)$, that is the real symmetric matrix  with rows and columns indexed by $\N_r^n$ and the following entrywise definition: 
\[ 
(\M_r(\y))_{\beta,\gamma} := \ell_\y(\x^{\beta + \gamma}) \,, \quad
\forall \beta, \gamma \in \N_r^n \,. 
\]

Given a polynomial $q \in \R[\x]$, we also associate to $\y$ 
the so-called {\it localizing matrix}, that is the real symmetric matrix $\M_r(q \, \y)$ with rows and columns indexed by $\N_r^{n}$ and the following entrywise definition: 
\[ 
(\M_r(q \, \y))_{\beta, \gamma} := \ell_\y(q(\x) \, \x^{\beta + \gamma}) \,, \quad
\forall \beta, \gamma \in \N_r^n \,. 
\]

In the sequel, we propose a method to compute an outer approximation of $\X^\infty$ as the super level set of a polynomial. Given its degree $2 r$, such a polynomial may be produced using standard numerical optimization tools. Under certain assumptions, the resulting outer approximation can be as tight as desired and converges (with respect to the $L_1$ norm on $\X$) to the set $\X^\infty$ as $r$ tends to infinity.

\section{Primal-Dual Infinite-dimensional LP}
\label{sec:infLP}
\subsection{Forward Reachable Sets and Liouville's Equation}
\label{sec:primalLPreach}
For a given $T \in \N_0$ and a measure $\mu_0 \in \mathcal{M}_+(\X^0)$, let us define the measures $\mu_1, \dots, \mu_T,\mu \in \mathcal{M}_+(\X)$ as follows:
\begin{equation}
\label{eq:measures}
\begin{aligned}
\mu_{t+1}  &:=  f_{\#} \mu_t = f_{\#}^{t+1}\mu_0 \,, \ t = 0,\dots, T-1 \,, \\
\nu  & := \sum\limits_{t=0}^{T-1} \mu_t = \sum\limits_{t=0}^{T-1}  f_{\#}^t \mu_0 \,.
\end{aligned}
\end{equation}

The measure $\nu$ is a (discrete-time) occupation measure: if $\mu_0=\delta_{x_0}$ is the Dirac measure at $x_0 \in \X^0$ then $\mu_t=\delta_{x_t}$ and $\nu = \delta_{x_0}+\delta_{x_1}+\cdots+\delta_{x_{T-1}}$ , i.e.~$\nu$ measures the time spent by the state trajectory in any subset of $\X$ after $T$ iterations, if initialized at $x_0$.

\begin{lemma}
\label{th:liouville1}
For any $T \in \N_0$ and $\mu_0 \in \mathcal{M}(\X^0)$, there exist $\mu_T,\nu \in \mathcal{M}(\X)$ solving the discrete Liouville Equation:
\begin{equation}
\label{eq:liouvilleT}
\mu_T + \nu = f_\# \nu + \mu_0 \,.
\end{equation}
\end{lemma}
\begin{proof}
Follows readily from the definitions given in~\eqref{eq:measures}.
\end{proof}

Now let
\[
\Y^0 := \X^0, \quad \Y^t := f^t(\X^0)  \backslash  \X^{t-1}, \:\:t=1,\ldots,T.
\]
Note that the RS defined in Section \ref{sec:reachset}  is equal to 
\[
\X^T = \cup_{t=0}^T \Y^t.
\]
Further results involve statements relying on  the following technical assumption:
\begin{assumption}
\label{hyp:finitesum}
$\lim_{T\to\infty} \sum_{t=0}^{T} t \vol \Y^t < \infty$.
\end{assumption}
This assumption seems to be strong or unjustified at that point. Moreover, if we do not know if the assumption is satisfied a priori, there is an a posteriori validation based on duality theory. {Thanks to this validation procedure, we could check that the assumption was satisfied in all the examples we processed. Moreover, we were not able to find a discrete-time polynomial system violating this assumption.} 
We will explain this later on {in Section~\ref{sec:sdp}, with Theorem~\ref{th:lpstrength} (see also Remark~\ref{rk:lpstrength})}.
Before that, {we prove in Lemma~\ref{th:liouvilleinf}
} that equation~\eqref{eq:liouvilleT} holds when $\mu_T=\lambda_{\X^T}$, the restriction of the Lebesgue measure over the RS. {We rely on the auxiliary result from~\cite[Lemma~4.1]{MHL15image}:
\begin{lemma}
\label{th:pullback}
Let $\Sb, \B \subseteq \X$ be such that $f(\Sb) \subseteq \B$. Given a measure $\mu_1 \in \mathcal{M}_+(\B)$, there is a measure $\mu_0 \in \mathcal{M}_+(\Sb)$
such that $f_\#\mu_0 = \mu_1$ if and only if there is no continuous function $v \in \mathcal{C}(\B)$
such that $v(f(\x)) \geq 0$ for all $\x \in \Sb$ and $\int_{\B} v(\y) d\mu_1(\y) < 0$.
\end{lemma}
}
\begin{lemma}
\label{th:liouvilleinf} 
For any $T \in \N_0$, there exist $\mu_0^T \in \mathcal{M}_+(\X^0)$ and $\nu^T \in \mathcal{M}_+(\X)$ such that the restriction of the Lebesgue measure over $\X^{T}$ solves the discrete Liouville Equation:
\begin{equation}
\label{eq:liouvilleTmax}
\lambda_{\X^T} + \nu^T = f_\# \nu^T + \mu_0^T \,.
\end{equation}
In addition, if Assumption~\ref{hyp:finitesum} holds, then  
there exist $\mu_0 \in \mathcal{M}_+(\X^0)$ and $\nu \in \mathcal{M}_+(\X)$ such that the restriction of the Lebesgue measure over $\X^{\infty}$ solves the discrete Liouville Equation:
\begin{equation}
\label{eq:liouvilleinf}
\lambda_{\X^{\infty}} + \nu = f_\# \nu + \mu_0 \,.
\end{equation}
\end{lemma}
\begin{proof}
We first show that for each $t = 1, \dots, T$ there exist a measure $\nu_{t} \in \mathcal{M}_+(\X)$ and $\mu_{0, t} \in \mathcal{M}_+(\X^0)$ such that 
\begin{equation}
\label{eq:liouvillet}
\lambda_{\Y^t} + \nu_t = f_\# \nu_t + \mu_{0, t}.
\end{equation}
%
It follows from {Lemma~\ref{th:pullback}} that there exists a measure $\mu_{0, t} \in \mathcal{M}_+(f^{-t}(\Y^t))$ such that $\lambda_{\Y^t} = f^t_\# \mu_{0, t}$ 
(with the notations $f \leftarrow f^t $, $\Sb \leftarrow  f^{-t}(\Y^t)$, $\B \leftarrow  \X$, $\mu_1 \leftarrow  \lambda_{\Y^t}$ and $\mu_0 \leftarrow  \mu_{0, t}$). 
{Indeed, by the change-of-variable formula recalled in Section~\ref{sec:pre}, one has $\int_\X v(\x) d \lambda_{\Y^t} (\x) = \int_{\Y^t} v(\x) d \lambda_{\Y^t} (\x) = \int_{f^{-t}(\Y^t)} v(f^{t}(\x)) d \mu_{0, t} (\x)$. Thus, it is impossible to find a function $v \in \mathcal{C}(\X)$ such that
$v(f^{t}(\x)) \geq 0$ for all $\x \in f^{-t}(\Y^t)$ while satisfying $\int_\X v(\x) d \lambda_{\Y^t} (\x) < 0$}.
Since $f^{-t}(\Y^t) \subseteq \X^0$, it implies that the measures $\mu_{0, t}$ and $\nu_t := \sum_{i=0}^{t-1} f^i_\# \mu_{0, t}$ satisfy $\mu_{0, t} \in \mathcal{M}_+(\X^0)$, $\nu_{t} \in \mathcal{M}_+(\X)$ and Equation~\eqref{eq:liouvillet}.

Letting $\mu_{0, 0}:=\lambda_{\X^0}$, we now prove the first claim of the Lemma by showing that the measures
\[
\nu^{T} := \sum_{t=0}^{T} \nu_t \in \mathcal{M}_+(\X)
\]
and
\[
\mu_0^{T} := \sum_{t=0}^{T} \mu_{0, t} \in \mathcal{M}_+(\X^0)
\]
satisfy
\begin{equation}
\label{eq:liouvilleXT}
\lambda_{\X^{T}} + \nu^{T} =  f_\# \nu^{T} + \mu_0^{T} \,.
\end{equation} 
Since $\X^T = \bigcup_{t=0}^T \Y^t$ and the $\Y^t$ are disjoint, one can write $\lambda_{\X^{T}} = \sum_{t=0}^T \lambda_{\Y^t}$. This decomposition together with Equation~\eqref{eq:liouvillet} allows to show the first claim.

Now, we prove the second claim of the Lemma by showing that the respective limits of the measure sequences $(\lambda_{\X^{T}})_T$, $(\nu^{T})_T$ and $(\mu_0^{T})_T$ exist as $T \to \infty$, and that the respective limits $\lambda_{\X^{\infty}}$, $\nu$ and $\mu_0$ satisfy Equation~\eqref{eq:liouvilleinf}:
\begin{itemize}
\item The limit of $\lambda_{\X^{T}}$ exists and is equal to $\lambda_{\X^{\infty}}$ by definition of $\X^{\infty}$;
\item The limit of $\mu_0^{T}$ exists since $\|\mu_0^{T}\|_{\text{TV}} = \sum_{t=0}^{T} \int_{\X^0}  \mu_{0, t} = \sum_{t=0}^{T}  \int_{\Y^t}  \lambda_{\Y^t} = \sum_{t=0}^{T} \vol \Y^t \leq \vol \X^{\infty} \leq \vol \X < \infty$. Therefore there is a subsequence
which converges to a certain $\mu_0 \in \mathcal{M}_+(\X^0)$ for the weak-star topology.
\item The limit of $\nu^{T}$ exists since $\|\nu^{T}\|_{\text{TV}} = \sum_{t=0}^T \int_{\X} \nu_t =   \sum_{t=0}^T \int_{\Y_ t}  \sum_{i= 0}^{t-1}   f^i_{\#}\mu_ {0,t} = \sum_{t=0}^{T} t \vol \Y^t < \infty$ by Assumption \ref{hyp:finitesum}. Therefore, there is a subsequence which converges to a certain $\nu \in \mathcal{M}_+(\X)$ for the weak-star topology. 
\end{itemize}
Finally, taking the limit to infinity of both sides of Equation~\eqref{eq:liouvilleXT} yields the initial claim.
\end{proof}

\begin{remark}
In Lemma~\ref{th:liouvilleinf}, 
the measure $\mu_0^T$ (resp.~$\mu_0$) can be thought as distribution of mass for the initial states of trajectories reaching $\X^T$ (resp.~$\X^\infty$) but it has a total mass which is not required to be normalized to one. 

The mass of $\nu^{T}$ measures the volume averaged w.r.t. $\mu_0$ occupied by state trajectories reaching  
$\X^T$ after $T$ iterations, by contrast with the mass of $\lambda_{\X^{T}}$ which measures the volume of $\X^T$.

The mass of $\nu$ measures the volume averaged w.r.t. $\mu_0$ occupied by state trajectories reaching  the RS $\X^{\infty}$, by contrast with the mass of $\lambda_{\X^{\infty}}$ which measures the exact RS volume.
\end{remark}

\subsection{Primal Formulation}
\label{sec:primalLP}

To approximate the set $\X^\infty$, one considers the infinite-dimensional linear programming (LP) problem, for any $T \in \N_0$:
\begin{equation}
\label{eq:lpmeasure}
\begin{aligned}
p^T := \sup\limits_{\mu_0, \mu, \hat{\mu}, \nu, a} \quad  & \int_\X  \mu  \\			
\text{s.t.} 
\quad & \int_\X  \nu + a  =   T \vol \X \,,\\
\quad & \mu + \nu  =  f_\# \nu + \mu_0 \,,\\
\quad & \mu + \hat{\mu}  =   \lambda_\X \,,\\
\quad & \mu_0 \in  \mathcal{M}_+(\X^0) \,, \quad  \mu, \hat{\mu}, \nu \in  \mathcal{M}_+(\X) \,, \quad a \in \R_+ \,.\\
\end{aligned}
\end{equation}
The first equality constraint ensures that the mass of the occupation measure $\nu$ is bounded (by $T \vol \X$). 
{The second one ensures that Liouville's equation is satisfied by the measures $\mu_0$, $\nu$ and $\mu$, as in Lemma~\ref{th:liouvilleinf}. The last one ensures that $\mu$ is dominated by the restriction of the Lebesgue measure on $\X$ implying that the mass of $\mu$ (and thus the optimal value $p^T$) is bounded by $\vol \X$. The next result explains how the solution of LP~\eqref{eq:lpmeasure} relates to $\lambda_{\X^\infty}$, the restriction of the Lebesgue measure to the RS.
}
\begin{lemma}
\label{th:lpmeasure}
For any $T \in \N_0$, LP~\eqref{eq:lpmeasure} admits an optimal solution $(\mu_0^*, \mu^*, \hat{\mu}^*, \nu^*, a^*)$ such that
$\mu^*=\lambda_{\Sb^T}$ for some set $\Sb^T$ satisfying $\X^T \subseteq \Sb^T \subseteq \bar{\X}^\infty$ and $\vol \Sb^T = p^T$. 

In addition if Assumption~\ref{hyp:finitesum} holds then there exists $T_0 \in \N$ such that for all $T \geq T_0$ one has 
$\Sb^T = \bar{\X}^\infty$, LP~\eqref{eq:lpmeasure}  has a unique optimal solution with $\mu^*=\lambda_{\X^\infty}$ and $p^T =  \vol \X^\infty$.
\end{lemma}

\begin{proof}
\label{pr:lpmeasure}
Let $T \in \N_0$.
First we show that the feasible set of LP~\eqref{eq:lpmeasure} is nonempty and compact, and that LP~\eqref{eq:lpmeasure} has at least one optimal solution $(\mu_0^*, \mu^*, \hat{\mu}^*, \nu^*, a^*)$.
The feasible set of LP~\eqref{eq:lpmeasure} is nonempty as $(\mu_0, \mu, \hat{\mu}, \nu, a) = (0, 0, \lambda_\X, 0, T \vol \X)$ is a feasible solution. Let us consider the sequences of measures $(\mu_{0n})_n$, $(\mu_n)_n$, $(\hat{\mu}_n)_n$, $(\nu_n)_n$ and the nonnegative real sequence $(a_n)_n$ 
such that each $(\mu_{0n},  \mu_n, \hat{\mu}_n, \nu_n, a_n)_n$ is feasible for LP~\eqref{eq:lpmeasure}. One has:
\begin{itemize}
\item $\| \mu_n \|_{\text{TV}} + \|  \hat{\mu}_n \|_{\text{TV}} = \vol \X < \infty$ (as $\X$ is bounded), which implies that  $\| \mu_n \|_{\text{TV}}$ and $\|  \hat{\mu}_n \|_{\text{TV}}$ are both bounded;
\item $\| \mu_{0n} \|_{\text{TV}}  = \| \mu_n \|_{\text{TV}} < \infty$, which implies that $\| \mu_{0n} \|_{\text{TV}}$ is bounded;
\item $\| \nu_n \|_{\text{TV}}$ and $a_n$ are both bounded by $ T \vol \X$.
\end{itemize}
Thus, the feasible set of the LP~\eqref{eq:lpmeasure} is bounded for the  weak-star topology. Now we show that this set is closed.
Assume that the sequences of measures respectively converge weakly-star to $\mu_0$, $\mu$, $\hat{\mu}$,  $\nu$ and the sequence of nonnegative real numbers converge to $a$.
For each $A \in \mathcal{B}(\X)$, one has $\nu (f^{-1} A) \to \nu_n (f^{-1} (A))$ and hence $f_\# \nu_n  \to f_\# \nu$. Thus,  $(\mu_0, \mu, \hat{\mu}, \nu, a)$ is also feasible for  LP~\eqref{eq:lpmeasure}, proving that the feasible set of the LP~\eqref{eq:lpmeasure} is closed in the metric including the weak-star topology.
The existence of an optimal solution  $(\mu_0^*, \mu^*, \hat{\mu}^*, \nu^*, a^*)$ follows from the fact that LP~\eqref{eq:lpmeasure} has a linear objective function with a weak-star compact feasible set. 

By Lemma~\ref{th:liouvilleinf}, there exist $\mu_0^T \in \mathcal{M}_+(\X^0)$ and $\nu^T \in \mathcal{M}_+(\X)$ such that the restriction of the Lebesgue measure over $\X^{T}$ solves the discrete Liouville Equation, i.e. $\lambda_{\X^T} + \nu^T = f_\# \nu^T + \mu_0^T$ and $\int_\X \nu^T = \sum_{t=0}^T t \vol \Y^t  \leq T \vol \X$. With $a^T = T \vol \X - \sum_{t=0}^T t \vol \Y^t$ and $\hat{\mu}^T = \lambda_{\X} - \lambda_{\X^T}$, it follows that $(\mu_0^T, \mu^T, \hat{\mu}^T, \nu^T, a^T)$  is feasible for LP~\eqref{eq:lpmeasure}.

Next, given any feasible solution $(\mu_0', \mu', \hat{\mu}', \nu', a')$ of LP~\eqref{eq:lpmeasure}, we show that the support of $\mu'$ is included in $\bar{\X}^\infty$. 
Using the Liouville Equation and the fact that $\mu_0'$ is supported on $\X^0 \subseteq \X^\infty$, one has:
\begin{align*}
\mu'(\X^\infty) + \mu'(\X \backslash \X^\infty)
= \mu'(\X)   = \mu_0'(\X^0) = \mu_0'(\X^\infty) 
= \mu'(\X^\infty) + \nu'(\X^\infty) - f_\# \nu'(\X^\infty) 
 \,,
\end{align*}
which proves that $\mu'(\X \backslash \X^\infty)=  \nu'(\X^\infty) - f_\# \nu'(\X^\infty) $.
Since $\X^\infty \subseteq f^{-1} (\X^\infty)$, one has $\nu'(\X^\infty) \leq \nu'(f^{-1} (\X^\infty)) = f_\# \nu'(\X^\infty) $, implying that $\mu'(\X \backslash \X^\infty) = 0$. Thus, the support of $\mu'$ is included in $\bar{\X}^\infty$. 

Now we assume that $\sum_{t=0}^{\infty} t \vol \Y^t > T \vol \X$.  Hence, there exist a minimal integer $T_2 \in \N_0$ and measures $\mu_0^{T_2} \in \mathcal{M}_+(\X^0)$, $\nu^{T_2} \in \mathcal{M}_+(\X)$ such that $T < T_2$, $\lambda_{\X^{T_2}} + \nu^{T_2} = f_\# \nu^{T_2} + \mu_0^{T_2}$ and $T \vol \X < \int_\X \nu^{T_2}$. 
Therefore, for some $T_1 \in \N_0$ with $T_1 \geq T$, there exist some sets $\Sb_0, \dots,\Sb_{T_1} \subseteq \X$ such that $\Sb_0 = \Y^0, \dots, \Sb_{T} = \Y^T, \Sb_{T+1} \subseteq \Y^{T+1}, \dots, \Sb_{T_1} \subseteq \Y_{T_1}$ and $\sum_{t=0}^{T_1} t \vol \Sb_t = T \vol \X$. By defining $\Sb^T := \bigcup_{t=0}^{T_1} \Sb_t$, one has $\X^{T} \subseteq \Sb^T \subset \X^{T_2} \subseteq \X^\infty$.  

As for Lemma~\ref{th:liouvilleinf}, one proves that for all $t=0, \dots, T_1$, there exist $\mu'_{0,t} \in \mathcal{M}_+(\X^0)$, $\nu'_t \in \mathcal{M}_+(\X)$  such that $\lambda_{\Sb_t} + \nu'_t = f_\# \nu'_t + \mu'_{0,t}$.  

By superposition, the measures $\lambda_{\Sb^T}$, $\mu_0^{T_1} := \sum_{i=0}^{T_1} \mu'_{0,t}$ and $\nu^{T_1} := \sum_{i=0}^{T_1} \nu'_t$   satisfy $\lambda_{\Sb^T} + \nu^{T_1} = f_\# \nu^{T_1} + \mu_0^{T_1}$ and $\int_\X \nu^{T_1} = \sum_{t=0}^{T_1} t \vol \Sb_t = T \vol \X$. 

This implies that $\lambda_{\Sb^T}$ is feasible for LP~\eqref{eq:lpmeasure}. Then one proves, exactly as in the proof of~\cite[Theorem 3.1]{HLS09vol}, that the quintuple $(\mu_0^{T_1}, \lambda_{\Sb^T}, \lambda_{\X \backslash \Sb^T}, \nu^{T_1}, 0)$ is optimum, yielding the optimal value $p^T =  \vol \Sb^T$.

Under Assumption~\ref{hyp:finitesum}, Lemma~\ref{th:liouvilleinf} implies that there exist $\mu_0 \in \mathcal{M}_+(\X^0)$ and $\nu \in \mathcal{M}_+(\X)$ such that the restriction of the Lebesgue measure over $\X^{\infty}$ solves the discrete Liouville's Equation~\eqref{eq:liouvilleinf}. In addition, there exists $T_0 \in \N$ such that for all $T \geq T_0$, $\sum_{t=0}^{T} t \vol \Y^t \leq T \vol \X$, thus $\int_\X \nu \leq T \vol \X $. This implies that $\lambda_{\X^\infty}$ is feasible for LP~\eqref{eq:lpmeasure}.
With $\mu_0$ and $\nu$ as in Lemma~\ref{th:liouvilleinf}, we define $\mu_0^* := \mu_0$, $\mu^* := \lambda_{\X^\infty}$, $\hat{\mu}^* := \lambda_\X - \lambda_{\X^\infty}$, $\nu^* := \nu$ and $a^* := T \vol \X - \int_\X \nu$ and show as in the proof of~\cite[Theorem 3.1]{HLS09vol}, that $(\mu_0^*, \mu^*, \hat{\mu}^*, \nu^*, a^*)$ is optimum with unique $\mu^*$. This yields the optimal value {$p^T = \int_\X \lambda_{\X^\infty} = \vol \bar{\X}^\infty = \vol \X^\infty$, where the last equality comes from Assumption~\ref{hyp:closure}}. 
\end{proof}

From now on, we refer to  $\Sb^T$ as the support  of the optimal solution  $\mu^*$ of LP~\eqref{eq:lpmeasure} which satisfies the condition of Lemma~\ref{th:lpmeasure}, i.e.~$\X^T \subseteq \Sb^T \subseteq \bar{\X}^\infty$.

In the sequel, we formulate LP~\eqref{eq:lpmeasure} as an infinite-dimensional conic problem
on appropriate vector spaces. By construction, a feasible solution of problem~\eqref{eq:lpmeasure} satisfies:
\begin{align}
\int_\X  \nu (d \x) + a  & =  \int_\X  T \lambda (d \x) \,,  \label{eq:testu} \\
\int_\X v(\x) \, \mu (d \x) + \int_\X v(\x) \, \nu (d \x) & =  \int_\X  v(f(\x)) \,  \nu (d \x) + \int_{\X^0} v(\x) \, \mu_0 (d \x)  \,, \label{eq:testv} \\
\int_\X w(\x) \, \mu(d \x) + \int_\X w(\x) \, \hat{\mu} (d \x) & =  \int_\X w(\x) \, \lambda (d \x)  \,, \label{eq:testw}
\end{align}
for all continuous test functions $v, w \in \mathcal{C}(\X)$.

Then, we cast problem~\eqref{eq:lpmeasure} as a particular instance of a primal LP in the canonical form given in~\cite[7.1.1]{alexander2002course}:
\begin{equation}
\label{eq:lpprimal}
\begin{aligned}
p^T = \sup\limits_{x} \quad & \langle x, c \rangle_1 \\
\text{s.t.} \quad & \mathcal{A} \, x = b  , \\
\quad & x \in E_1^+ , \\
\end{aligned}
\end{equation}
with
\begin{itemize}
\item the vector space $E_1 := \mathcal{M}(\X^0) \times \mathcal{M}(\X)^3 \times \R $ with its cone $E^+_ 1$ of non-negative elements;
\item the vector space $F_1 := \mathcal{C}(\X^0) \times \mathcal{C}(\X)^3 \times \R $;
\item the duality $\langle \cdot , \cdot \rangle_1 : E_1 \times F_1 \to \R$, given by the integration of continuous functions against Borel measures, since $E_1 = F_1'$ is the dual of $F_1$;
\item the decision variable $x := (\mu_0, \mu, \hat{\mu}, \nu, a) \in E_1$ and the reward $c := (0, 1, 0, 0, 0) \in F_1$;
\item $E_2 := \R \times \mathcal{M}(\X)^2$, $F_2 := \R \times  \mathcal{C}(\X)^2$ and the right hand side vector $b := (T \vol \X, 0, \lambda) \in E_2=F_2'$;
\item the linear operator $\mathcal{A} : E_1 \to E_2$ given by
\[
\mathcal{A} \, (\mu_0, \mu, \hat{\mu}, \nu, a) := \left[\begin{array}{cc}
\int_\X \nu + a \\
\mu + \nu - f_\# \nu - \mu_0 \\
\mu +\hat{\mu} 
\end{array}\right].
\]
\end{itemize}
Note that both spaces $E_1$, $E_2$ (resp. $F_1$, $F_2$) are equipped with the weak topologies $\sigma(E_1, F_1)$, $\sigma(E_2, F_2)$ (resp. $\sigma(F_1, E_1)$, $\sigma(F_2, E_2)$) and $\sigma(E_1, F_1)$ is the weak-star topology (since $E_1 = F_1'$). Observe that $\mathcal{A}$ is continuous with respect to the weak topology, as $\mathcal{A}'(F_2) \subset F_1$.

\subsection{Dual Formulation}
\label{sec:dualLP}

Using the same notations, the dual of the primal LP~\eqref{eq:lpprimal} in the canonical form given in~\cite[7.1.2]{alexander2002course} reads:
\begin{equation}
\label{eq:lpdual}
\begin{aligned}
d^T = \inf\limits_{y} \quad & \langle b, y \rangle_2 \\
\text{s.t.} \quad & \mathcal{A}' \, y  -c \, \in F_1^+ \,,\\
\end{aligned}
\end{equation}
with
\begin{itemize}
\item the dual variable $y := (u,v,w) \in F_2$;
\item the (pre)-dual cone $F^+_1:=\R^+ \times \mathcal{C}_+(\X)^2$, whose dual is $E_1^+$;
\item the duality pairing $\langle \cdot , \cdot \rangle_2 : E_2 \times F_2 \to \R$, with $E_2 = F_2'$;
\item the adjoint linear operator $\mathcal{A}' : F_2 \to F_1$ given by
\[
\mathcal{A}' \, (u, v, w) := \left[\begin{array}{c}
-v \\
w + v \\ 
w \\ 
u + v - \, v \circ f
\end{array}\right].
\]
\end{itemize}

Using our original notations, the dual LP of problem~\eqref{eq:lpmeasure} then reads:
\begin{equation}
\label{eq:lpcont}
\begin{aligned}
d^T := \inf\limits_{u, v, w} \quad & \int_\X (w(\x) + T u) \, \lambda_\X (d  \x)  \\			
\text{s.t.} 
\quad & v(\x) \geq 0 , \quad \forall \x \in \X^0  , \\
\quad & w(\x) \geq 1 + v(\x) , \quad \forall \x \in \X  , \\
\quad & w(\x) \geq 0 , \quad \forall \x \in \X  , \\
\quad & u + v(f(\x)) \geq v(\x) , \quad \forall \x \in \X  , \\
\quad & u \geq 0  , \\
\quad & u \in \R \,, \quad v, w \in \mathcal{C}(\X)  . \\
\end{aligned}
\end{equation}

\begin{theorem}
\label{th:zerogap}
For a fixed $T \in \N_0$, there is no duality gap between primal LP~\eqref{eq:lpmeasure} and dual LP~\eqref{eq:lpcont}, i.e. $p^T = d^T$ and there exists a minimizing sequence $(u_k,v_k,w_k)_{k \in \N}$ for the dual LP~\eqref{eq:lpcont}.

In addition, if $u_k = 0$ for some $k\in\N$, then Assumption~\ref{hyp:finitesum} holds and  $p^T = d^T = \vol \X^\infty$. 
\end{theorem}
\begin{proof}
\label{pr:zerogap}
As in~\cite[Theorem 3]{KHJ13mci}, the zero duality gap follows from infinite-dimensional LP duality theory 
(for more details, see~\cite{alexander2002course}).
The feasible set of the LP~\eqref{eq:lpmeasure} is nonempty in the metric
inducing the weak-star topology on $\mathcal{M}(\X^0) \times \mathcal{M}(\X)^3 \times \R$
since $(0, 0, \lambda_\X, 0, T \vol \X)$ is a trivial feasible solution.
As shown in Lemma~\ref{th:lpmeasure}, the feasible set of the LP~\eqref{eq:lpmeasure} is bounded for the same metric. 
Hence, the first claim follows from nonemptiness and boundedness of the feasible set of the primal LP~\eqref{eq:lpmeasure}.

Now, let us assume that there exists a minimizing sequence $(u_k,v_k,w_k)_{k \in \N}$ of the dual LP~\eqref{eq:lpcont} such that the dual variable $u_k$ is equal to 0 for some $k \in \N$. 
This implies that the corresponding constraint in the primal LP~\eqref{eq:lpmeasure} is not saturated, i.e.~$\int_\X \nu < T \vol \X$, for any solution $(\mu_0, \mu, \hat{\mu}, \nu, a)$. 
Let us show that Assumption~\ref{hyp:finitesum} holds. Otherwise, by contradiction one would have $\sum_{t=0}^\infty t \vol \Y^t > T \vol \X$. In the proof of Lemma~\ref{th:lpmeasure}, we proved that this strict inequality implies the existence of a set $\Sb^T$ and an optimal solution $(\mu_0^{T_1}, \lambda_{\Sb^T}, \lambda_{\X \backslash \Sb^T}, \nu^{T_1}, 0)$ for the primal LP~\eqref{eq:lpmeasure}, such that $\int_\X \nu^{T_1} = T \vol \X$, yielding a contradiction. Eventually, Lemma~\ref{th:lpmeasure} implies that $p^T = d^T = \vol \X^\infty$.
\end{proof}
{
\begin{remark}
\label{rk:dualLyapunov}
When $u = 0$, the first and third inequalities satisfied by $-v$ in dual LP~\eqref{eq:lpcont} can be seen as a discrete-time analogue of the conditions satisfied by the barrier certificate in~\cite{Prajna04}, that is $-v \leq 0$ on $\X^0$ and $-v \circ f \leq -v$ on $\X$, the latter one being similar to the barrier condition $\nabla (-v) \cdot f \leq -v$.
\end{remark}
}
\section{Primal-Dual Hierarchies of SDP Approximations}
\label{sec:sdp}

\subsection{Primal-Dual Finite-dimensional SDP}
\label{sec:finitesdp}
For each $r \geq r_{\min} := \max \{r^0_1, \ldots, r^0_{m^0}, r_1, \ldots, r_m\}$,
let $\y_0 = (y_{0 \beta})_{\beta \in \N_{2 r}^n}$ be the finite sequence of moments up to degree $2 r$
of the measure $\mu_0$. Similarly, let $\y$, $\hat{\y}$ and $\z$ stand for the sequences of moments up to degree $2 r$, respectively associated with $\mu$, $\hat{\mu}$ and $\nu$.
The infinite primal LP~\eqref{eq:lpmeasure} can be relaxed with the following semidefinite program:
\begin{equation}
\label{eq:lprelax}
\begin{aligned}
p^T_r := \sup\limits_{\y_0, \y, \hat{\y}, \z, a} \quad & y_{0} \\			
\text{s.t.} 
\quad & z_{0} + a = T y^{\X}_0 \,, \\
\quad & y_{\beta} + z_{\beta} =  \ell_{\z}(f(\x)^\beta) + y_{0 \beta} \,,
\quad \forall \beta \in N_{2 r}^n \,, \\
\quad & y_{\beta} + \hat{y}_{\beta} = y^{\X}_\beta, \quad \forall \beta \in N_{2 r}^n \,, \\
\quad & \M_{r d -r_j^0} (g_j^0 \, \y_0) \succeq 0, \quad j = 0,\dots, m^0 \,, \\
\quad & 
\M_{r - r_j}(g_j \, \y) \succeq 0 \,, 
\M_{r - r_j} (g_j  \, \hat{\y}) \succeq 0 \,, 
\M_{r - r_j}(g_j \, \z) \succeq 0 \,, 
\quad j = 0,\dots, m \,,\\
\quad &  a \geq 0 \,.
\end{aligned}
\end{equation}
Consider also the following semidefinite program, which is a strengthening of the infinite dual LP~\eqref{eq:lpcont}
and also the dual of Problem~\eqref{eq:lprelax}:
\begin{equation}
\label{eq:lpstrength}
\begin{aligned}
d^T_r := \inf\limits_{u, v, w} \quad & \sum_{\beta \in \N_{2 r}^n} w_\beta z^\X_\beta + u T z^\X_0 \\			
\text{s.t.} 
\quad & v \in \Q_r^0 \,,\\
\quad & w -1-v \in \Q_r \,, \\
\quad & u + v \circ f - v \in \Q_{r d} \,, \\
\quad & w \in \Q_r \,,\\
\quad & u \in \R^+ \,, \\
\quad & v, w \in \R_{2 r}[\x] \,,
\end{aligned}
\end{equation}
{where $\Q_r^0$, $\Q_r$ (resp.~$\Q_{r d}$) are the $r$-truncated (resp.~$r d$) quadratic module respectively generated by $g_0^{0}, \dots, g_m^{m^0}$ and $g_0, \dots, g_m$, as defined in Section~\ref{sec:pre}.
}
\begin{theorem}
\label{th:lpstrength}
Let $r \geq r_{\min}$. Suppose that the three sets $\X^0$, $\Sb^T$ and $\X \backslash \Sb^T$ have nonempty interior. Then:
\begin{enumerate}
\item $p^T_r =d^T_r$, i.e. there is no duality gap between the primal SDP program~\eqref{eq:lprelax} and the dual SDP program~\eqref{eq:lpstrength}.
\item The dual SDP program~\eqref{eq:lpstrength} has an optimal solution $(u_r, v_r, w_r) \in \R \times \R_{2 r}[\x] \times \R_{2 r}[\x]$, and the sequence $(w_r + u_r T)$ converges to $\mathbf{1}_{\Sb^T}$ in $L_1$ norm on $\X$:
\begin{equation}
\label{eq:wdcvg}
\lim_{r \to \infty} \int |w_r(\x) + u_r T -\mathbf{1}_{\Sb^T} (\x)| \, \lambda_\X (d \x)  = 0.
\end{equation}
\item Defining the sets
\[
\X^T_r := \{ \x \in \X : v_r (\x) + u_r T \geq 0 \} \,,
\]
it holds that
\[
\X^T_r  \supseteq \X^T \,.
\]
\item In addition, if $u_r = 0$ then Assumption~\ref{hyp:finitesum} holds 
and the sequence $(w_r)$ converges to $\mathbf{1}_{{\bar{\X}^\infty}}$ in $L_1$ norm on $\X$. 
Defining the sets
\[
\X_r^\infty := \{ \x \in \X : v_r (\x) \geq 0 \} \,,
\]
its holds that
\[
\X^\infty_r  \supseteq \bar{\X}^\infty \supseteq \X^\infty \,.
\]
 and 
 \[
 \lim_{r \to \infty} \vol (\X^\infty_r \backslash \X^\infty) = \vol (\X^\infty_r \backslash \bar{\X}^\infty) = 0 \,.
 \]
\end{enumerate}
\end{theorem}
\begin{proof}
\label{pr:lpstrength}
~
\begin{enumerate}
\item 
Let $\mu^*:= \lambda_{\Sb^T}$, $\mu_0^*$ and $\nu^*$ be such that $\mu^* + \nu^* = f_\# \nu^* + \mu_0^*$ as in the proof of Lemma~\ref{th:lpmeasure},
and let $\hat{\mu}^* = \lambda_\X - \mu^*$ and $a^* := T \vol \X - \int_\X \nu^*$ so that $(\mu_0^*,\mu^*,\hat{\mu}^*, \nu^*,a^*)$ is feasible for LP (\ref{eq:lpmeasure}).
Given $r \geq r_{\min}$, let $\y_0$, $\y$, $\hat{\y}$ and $\z$ be the sequences of moments up to degree $2 r$
of $\mu_0^*$, $\mu^*$, $\hat{\mu}^*$ and $\nu^*$ respectively.  Then, as in the proof of the first item of Theorem 4.4 in~\cite{MHL15image}, the optimal set of the primal SDP program~\eqref{eq:lprelax} is nonempty and bounded and the result of~\cite{Trnovska05SDP} implies that there is no duality gap between the primal SDP program~\eqref{eq:lprelax} and dual SDP program~\eqref{eq:lpstrength}.
\item
The proof follows  by the same arguments as in the proof of the second item of Theorem 4.4 in~\cite{MHL15image}.
One first shows that $(\y_0, \y, \hat{\y}, \z)$ is strictly feasible for program~\eqref{eq:lprelax}. It comes from the fact the three sets $\X^0$, $\Sb^T$ and $\X \backslash \Sb^T$ have nonempty interior, respectively yielding the positive definiteness of $\M_{r - r_j^0}(g_j \, \y_0)$ for each $j = 0,\dots, m^0$, 
$\M_{r - r_j}(g_j \, \y)$, $\M_{r - r_j} (g_j \, \hat{\y})$ and $\M_{r - r_j}(g_j \, \z)$ for each $j = 0,\dots, m$. 

Then, we conclude that the dual SDP program~\eqref{eq:lpstrength} has an optimal solution $(u_r, v_r, w_r) \in \R \times \R_{2 r}[\x] \times \R_{2 r}[\x]$.

Now, one proves that there exists a sequence of polynomials $(w_k)_{k \in \N}
\subset \R[\x] $ such that $w_k(\x) \geq \mathbf{1}_{{\Sb^T}} (\x)$, for all $\x \in \X$ and such that
\begin{align}
\label{eq:overwdcvg}
\lim_{k \to \infty} \int  |w_k(\x) - \mathbf{1}_{\Sb^T} (\x) | \, \lambda_\X (d \x)  = 0.
\end{align}
Closedness of $\Sb^T$ implies that the indicator function $\mathbf{1}_{\Sb^T}$ is upper semi-continuous, so there exists a non-increasing sequence of bounded continuous functions $h_k : \X \to \R$ such that $h_k (\x) \downarrow \mathbf{1}_{\Sb^T} (\x)$, for all $\x \in \X$, as $k \to \infty$. 

By the Monotone Convergence Theorem~\cite{Ash72}, $h_k \to \mathbf{1}_{\Sb^T}$ for the $L_1$ norm. Let $\epsilon > 0$ be given. Using the Stone-Weierstrass Theorem~\cite{Stone48}, there exists a sequence of polynomials $(w_k')_{k \in \N} \subset \R[\x]$, such that $\sup_{\x \in \X} \mid w_k'(\x) - h_k(\x) \mid < \epsilon $, thus $\mid w_k'(\x) - \mathbf{1}_{\Sb^T}(\x) \mid < 2 \epsilon$.  With $u_k := 4 \epsilon$, the polynomial $w_k := w_k' + 2 \epsilon$ satisfies $0 < w_k - \mathbf{1}_{\Sb^T} < 4  \epsilon$ and $4 T \epsilon < w_k + u_k T - \mathbf{1}_{\Sb^T} < (4 + 4 T) \epsilon$, thus~\eqref{eq:overwdcvg} holds.

Finally, let us define the polynomials  $w_k'' := w_k + \epsilon$, $v_k := w_k - 1$ and prove that $(u_k, v_k'', w_k'')$ is a feasible solution of~\eqref{eq:lpstrength} for large enough $r \in \N$. 
The inequalities $w_k'' > w_k > \mathbf{1}_{\Sb^T}$ prove that $w_k'' \in \Q_r(\X)$, as a consequence of Putinar's Positivstellensatz~\cite[Section 2.5]{lasserre2009moments}.
Similarly, $w_k'' - v_k'' - 1 = \epsilon > 0$ which proves that $w_k'' - v_k'' - 1 \in \Q_r(\X)$. 
For each $\x \in \X$, one has  $0 < v''_k (\x) < 4 \epsilon$. The left inequality proves that $v_k'' \in \Q_r(\X^0)$.
Using both inequalities, one has for all $\x \in  \X$,  $u_k T + v''_k (f(\x)) - v_k''(\x) >  0$, so $u_k T + v_k'' \circ f - v_k''$ lies in $\Q_{r d}(\X)$. 
{}
\item Let $\x \in  \X^T$ and $(u_r, v_r, w_r) \in \R_{2 r}[\y] \times \R_{2 r}[\y]$ be an optimal solution of~\eqref{eq:lpstrength}. There exist $t \in \N$ such that $t \leq T$ and $\x \in \Y^t$. Thus, there exists {$\x^0 \in \X^0$ such that  $\x = f^t(\x^0)$.  By feasibility, $v_r(\x^0) \geq 0$ and by
  induction, one has $v_r(\x) + u_r t  \geq  v_r(\x^0) \geq 0$.} Since $u_r \geq 0$, one has $v_r(\x) + u_r T \geq  v_r(\x) + u_r t \geq 0$, yielding $\X^T_r  \supseteq \X^T$.
\item If the dual variable $u_r$ satisfies $u_r = 0$, then the corresponding constraint in the primal is not saturated, i.e.~$z_0 < T y_0^\X$, which implies as in the proof of Lemma~\ref{th:zerogap} that Assumption~\ref{hyp:finitesum} holds.
By Lemma~\ref{th:lpmeasure}, $\Sb^T = \bar{\X}^\infty$ and the sequence $(w_r)$ converges to $\mathbf{1}_{{\bar{\X}^\infty}}$ in $L_1$ norm on $\X$. For all $\x \in \X^{\infty}$, there exists {$\x^0 \in \X^0$ and $t \in \N$
  such that $\x = f^t(\x^0)$.  By feasibility, $v_r(\x^0) \geq 0$ implies that $v_r(\x) \geq v_r(\x^0) \geq 0$.} For all $\x \in \bar{\X}^{\infty}$, there exists a sequence $(\x_n) \subset \X^{\infty}$ converging to $\x$, and as above, one shows that for all $n \in \N$, one has $v_r(\x_n) \geq 0$. Continuity of $v_r$ implies that $v_r(\x) \geq 0$. This proves that $\X^\infty_r \supseteq \bar{\X}^\infty \supseteq \X^\infty$.
Finally, the proof of the convergence in volume is similar to the proof of {Theorem~3.2~(6) in~\cite{MHL15image}}.
\end{enumerate}
\end{proof}

\begin{remark}
\label{rk:lpstrength}
Theorem~\ref{th:lpstrength} states that one can over approximate the reachable states of the system after any arbitrary finite number of discrete-time steps (third item). 
In addition, Theorem~\ref{th:lpstrength} provides a sufficient condition to obtain a hierarchy of over approximations converging in volume to the RS (fourth item). {If $u_r$ = 0, then the sequence of optimal values of SDP~\eqref{eq:lpstrength} is nonincreasing and converges to the volume of the RS. If one defines the piecewise polynomial $\overline{v}_r := \min_{k \leq r} v_k$, then one shows as in~\cite[Theorem~1]{PMI11} that we obtain a  nonincreasing sequence of functions converging to the indicator function of the RS: one has $\overline{v}_r \downarrow \mathbf{1}_{\overline{\X}^\infty}$ almost everywhere, almost uniformly and in Lebesgue measure.}
\end{remark}
\subsection{Special Case: linear systems with ellipsoid constraints}
\label{sec:special}
%
Given $A \in \R^{n \times n}$, let us consider a discrete-time linear system $\x_{t+1} = \A \, \x_t$ with a set of initial constraints defined by the ellipsoid $\X^0 := \{\x \in \R^n : 1 \geq \x^T \, \V_0 \, \x \, \}$ with $\V_0\in \R^{n \times n}$ a positive definite matrix.

Similarly the set of state constraints is defined by the ellipsoid $\X := \{\x \in \R^n : 1 \geq \x^T \, \G \, \x \, \}$ with $\G  \in \R^{n \times n}$ a positive definite matrix. Since one has $\X^0 \subseteq \X$, it follows that $\V_0 \succeq \G$.

Then, one can look for a quadratic function $v(\x) := 1 - \x^T \, \V \, \x$, with $\V \in \R^{n \times n}$ a positive definite matrix solution of the following SDP optimization problem:
\begin{equation}
\label{eq:linear}
\begin{aligned}
\sup\limits_{\V \in \R^{n \times n}} \quad & \trace{\M \V} \\
\text{s.t.} 
\quad & \V_0 \succeq \V \succeq \A^T \V \A  \,,\\
\quad & \V \succ 0 \,,
\end{aligned}
\end{equation}
where $\M$ is the second-order moment matrix of the Lebesgue measure on $\X$, i.e. the matrix with entries
\[
(\M)_{\alpha,\beta} = y^{\X}_{\alpha+\beta}, \quad, \alpha, \beta \in \N^n,|\alpha|+|\beta|=2.
\]

Note that in this special case SDP~\eqref{eq:linear} can be retrieved from SDP~\eqref{eq:lpstrength} and one can over approximate the reachable set with the superlevel set of $v$ or $w - 1$:
\begin{lemma}
\label{th:linear}
SDP~\eqref{eq:linear} is equivalent to SDP~\eqref{eq:lpstrength} with $r := 1$, $u_r := 0$, $v(\x) := 1 - \x^T \, \V \, \x$ and $w(\x) = 1 + v(\x)$. Thus, one has:
\[
\{ \x \in \X : v (\x) \geq 0 \} = \{ \x \in \X : w (\x) \geq 1 \} \supseteq \X^\infty \,.
\]
\end{lemma}
\begin{proof}
%
The polynomial $v$ is nonnegative over $\X^0$ if and only if $\V_0 \succeq \V$. The ``if'' part comes from the fact that $\V_0 \succeq \V$ implies that $1 - \x^T \, \V \, \x \geq 1 - \x^T \, \V_0 \, \x \geq 0$, for all $\x \in \X^0$.
The other implication is a consequence of the $S$-Lemma~\cite{Yakubovich71}: there exists a nonnegative constant $c$ such that $1 - \x^T \, \V \, \x \geq c \, (1 - \x^T \, \V_0 \, \x)$ for all $\x \in \R^n$. It yields $1 \geq c$ when $\x = 0$. By defining $\y = (1, \x)$, one finds that $\y^T \, \begin{pmatrix} 1 - c & 0 \\ 0 & \V_0 - \V \end{pmatrix} \, \y \geq 0$ for all $\y \in \R^{n+1}$, which finally gives $\V_0 \succeq \V$. 

In addition, the polynomial $\x \mapsto v (\A \x) - v(\x)$ is nonnegative over $\X$ if and only if $\V \succeq \A^T \V \A$. The ``if'' part comes from the fact that $\V \succeq \A^T \V \A$ implies that $v (\A \x) - v(\x) = \x^T \, ( \V - \A^T \V \A)\, \x \geq 0$, for all $\x \in \X$.
The other implication follows from the $S$-Lemma: there exists a constant $c \geq 0$ such that $v (\A \x) - v(\x) \geq c \,  (1 - \x^T \, \G \, \x)$ or equivalently $\x^T \, (\V - \A^t \, \V \, \A) \, \x \geq c \, (1 - \x^T \, \G \, \x)$, for all $\x \in \R^n$. As before, this yields $0 \geq c$, thus $c = 0$ and $\V -  \A^t \, \V \, \A \succeq 0$, which finally gives $\V_0 \succeq  \V \succeq \A^T \V \A$.
%

Minimizing the integral (w.r.t.~the Lebesgue measure) of $w$ over $\X$ is equivalent to maximizing the
  integral  of the trace of the matrix $\x \, \x^T \V$ on $\X$.

\end{proof}
\if{
\subsubsection{Isometries}
\label{rk:isometries}
Let us consider a discrete-time polynomial system where the dynamics map $f$ is an isometry of $\X$, i.e.~for all $\x \in \X$, $f(\x) = \A \x + \b$, for some $n \times n$ orthogonal (or real unitary) matrix $\A$  and a vector $\b \in \R^n$. Here, we take $\b = \mathbf{0}$ to ensure that all states remain in $\X$. The group of unitary matrices being closed under multiplication, it follows that $\X^{\infty} = \bar{\X}^{\infty} = \bigcup_{t \in \N} \{ \, \A^t \X^0 : \X^0 \in \X^0 \, \}$.
}\fi
%
\if{
\subsection{Computational Considerations}
\label{sec:cmp}

\didier{Cette section est un peu gratuite, et incompréhensible pour un novice. Je l'enlèverais, quitte à la remettre si on nous le demande. }

Here, we analyze the cost of our method and explain how to overcome the potential computational blowup using either alternative formulations or correlation sparsity patterns.

As in~\cite[Section 5.1.2]{MHL15image}, we count the number of variables and constraints of the dual SDP program~\eqref{eq:lpstrength}. 
For $r \geq r_{\min}$, the number of variables $n_{\text{sdp}}$ of Problem~\eqref{eq:lpstrength} satisfies:
\[
n_{\text{sdp}} \leq 3 \binom{n + 2 r}{n} + \binom{n + 2 r d}{n} \,.
\]
Problem~\eqref{eq:lpstrength} involves $(m^0 + 1 + 2 m + 2)$ semidefinite constraints of size at most $\binom{n + r}{n}$ and $(m + 1)$ semidefinite constraints of size at most $\binom{n + r d}{n}$.

Note that when the degree $d$ of the dynamics is not small, the number of variables $n_{\text{sdp}}$ and constraints can both grow quickly. As already mentioned in in~\cite[Section 5.1.2]{MHL15image}, a potential workaround for large degrees consists of considering ``lifting'' variables $x_{n+1}, \dots, x_{2 n}$, respectively associated to the components $f_1, \dots, f_n$ of the dynamics, together with additional degree $d$ equality constraints $x_{n+1} = f_1(\x), \dots, x_{2 n} = f_n(\x)$. 
Thus, the minimal relaxation order of the corresponding semidefinite program is 
$r_{\min}' := \max \{\lceil\frac{d}{2}\rceil, r^0_1 , \ldots, r^0_{m^0}, r_1,
\ldots, r_m \}$.

For $r \geq r_{\min}'$, this program involves $n_{\text{sdp}}'$ variables satisfying:
\[
n_{\text{sdp}}' \leq 3 \binom{n + 2 r}{n} + \binom{2 n + 2 r}{2 n}  \,,
\]  
as well as $(m^0 + 1 + 2 m + 2)$ semidefinite constraints of size at most $\binom{n + r}{n}$ and $(m + 1)$ semidefinite constraints of size at most $\binom{2 n + d}{2 n}$.

Furthermore, one can also exploit the sparsity arising in the polynomial dynamics as well as in the polynomials describing general state (resp.~initial) state constraints. Following~\cite[Section 5.2]{MHL15image}, we can apply the sparse representation result of~\cite[Corollary 3.9]{Las06Sparse} while assuming that the variable index set $\{1, \dots, n\}$ is partitioned into $k$ disjoint sets $I_1, \dots, I_k$ such that:
\begin{enumerate}
\item The collection $\{I_1, \dots, I_k\}$ satisfies the so-called {\em running intersection property} (see Equation~\cite[(1.3)]{Las06Sparse} for more details).
\item For each $j=1,\dots,m^0$ (resp.~$j=1,\dots,m$), there exists some $k_j$ such that the polynomial $g_j^0$ (resp.~$g_j$) in the definition~\eqref{eq:defX0} (resp.~\eqref{eq:defX}) involves only variables $\{x_i \, | \, i \in I_{k_j} \}$.
\item In the definition~\eqref{eq:defX0} (resp.~\eqref{eq:defX}) of $\X^0$ (resp.~$\X$), we replace the inequality constraint $N^0 - \| \x \|_2^2 \geq 0$ (resp.~$N - \| \x \|_2^2 \geq 0$) by the $k$ inequality constraints $N_j^0 - \sum_{i \in I_j} x_i^2 \geq 0$ (resp.~$N_j - \sum_{i \in I_j} x_i^2 \geq 0$), with large enough $N_j^0$ (resp.~$N_j$), for each $j=1,\dots,k$.
\end{enumerate}
Then, under the above assumption, the collection of subsets $I_j' := I_j \cup \{n+1, \dots, 2 n \}$ ($j=1, \dots, k$) also satisfies the running intersection property.
}\fi
\section{Numerical Experiments}
\label{sec:bench}
Here, we present experimental benchmarks that illustrate our method. For a given positive integer $r$, we compute the polynomial solution $w_r$ of the dual SDP program~\eqref{eq:lpstrength}. This dual SDP is modeled using the {\sc Yalmip} toolbox~\cite{YALMIP} available within {\sc Matlab} and interfaced with the SDP solver {\sc Mosek}~\cite{mosek}. Performance results were obtained with an Intel Core i7-5600U CPU ($2.60\, $GHz) running under Debian 8.  

For all experiments, we could find an optimal solution of the dual SDP program~\eqref{eq:lpstrength} either by adding the constraint $u = 0$ or by setting $T = 100$. In the latter case, the optimal solution is such that $u_r \simeq 0$ and the polynomial solution $w_r$ is the same than in the former case, up to small numerical errors (in practice the value of $u_r$ is less than $1\text{e--}5$).
%
This implies that Assumption \ref{hyp:finitesum} is satisfied, i.e. the constraint of the mass of the occupation measure is not saturated, and yielding valid outer approximations of ${\X}^\infty$. The implementation is freely available on-line\footnote{\url{www-verimag.imag.fr/~magron/reachsdp.tar.gz}}.
%
%
\subsection{Toy Example}
\label{ex:toy}
First, let us consider the discrete-time polynomial system defined by
\begin{align*}
x_1^+ & := \frac{1}{2} (x_1 + 2 x_1 x_2)  \,, \\
x_2^+ & := \frac{1}{2} (x_2 - 2 x_1^3) \,, 
\end{align*}
with initial state constraints $\X^0 := \{ \x \in \R^2 : (x_1 - \frac{1}{2})^2 + (x_2 - \frac{1}{2})^2 \leq 4^{-2} \}$ and general state constraints within the unit ball $\X := \{ \x \in \R^2 : \| x \|_2^2  \leq 1 \}$.
On Figure~\ref{fig:toy}, we represent in light gray the outer approximations $\X^\infty_r$ of $\X^\infty$ obtained by our method, for increasing values of
the relaxation order $r$ (from $2 r=4$ to $14$). On each figure, the colored sets of points are obtained by simulation for the first $7$
iterates. More precisely, each colored set correspond to (under approximations of) the successive image sets $f(\X^0), \dots, f^7(\X^0)$ of the points obtained by uniform sampling of $\X^0$ under $f, \dots, f^7$ respectively. The set $\X^0$ is blue and the set $f^7(\X^0)$ is red, while intermediate sets take intermediate colors.
The dotted circle represents the boundary of the unit ball $\X$. Figure~\ref{fig:toy} shows that the over approximations are already quite tight for low degrees. 
\begin{figure}[!ht]
\centering
\subfigure[$2 r=4$]{
\includegraphics[scale=\sizesmallfig]{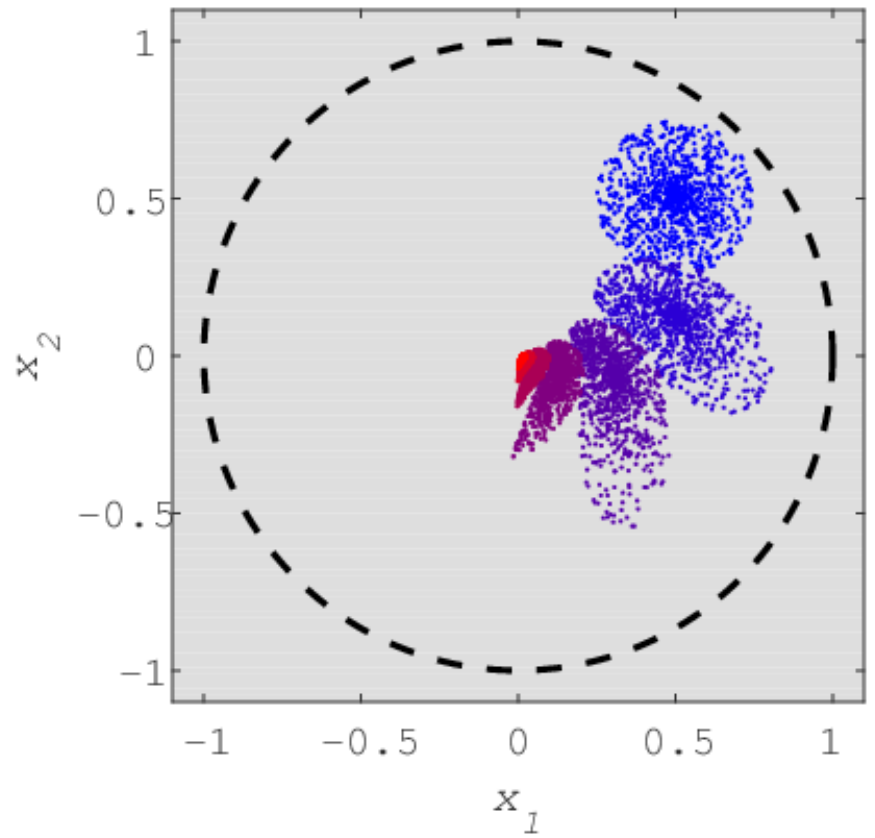}}
\subfigure[$2 r=6$]{
\includegraphics[scale=\sizesmallfig]{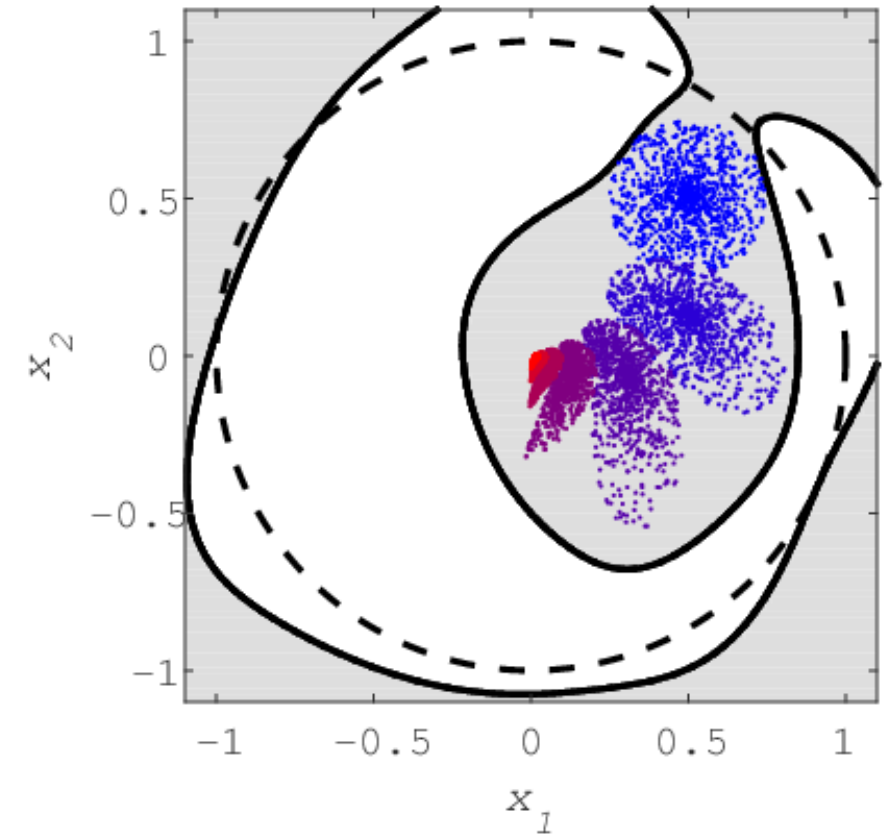}}
\subfigure[$2 r=8$]{
\includegraphics[scale=\sizesmallfig]{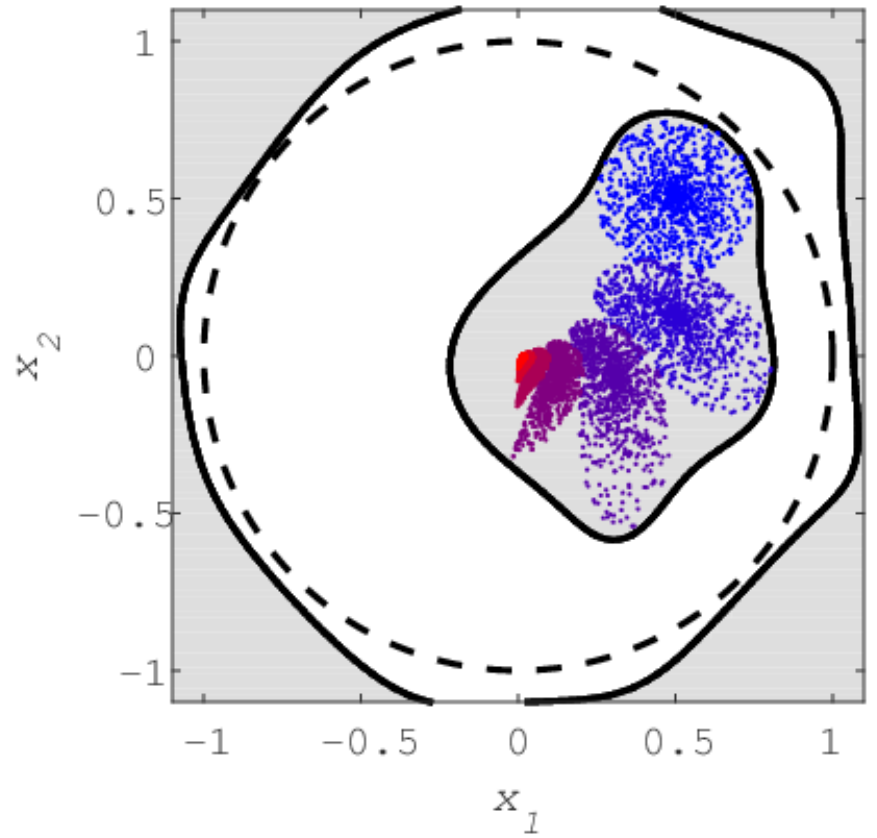}}\\
\subfigure[$2 r=10$]{
\includegraphics[scale=\sizesmallfig]{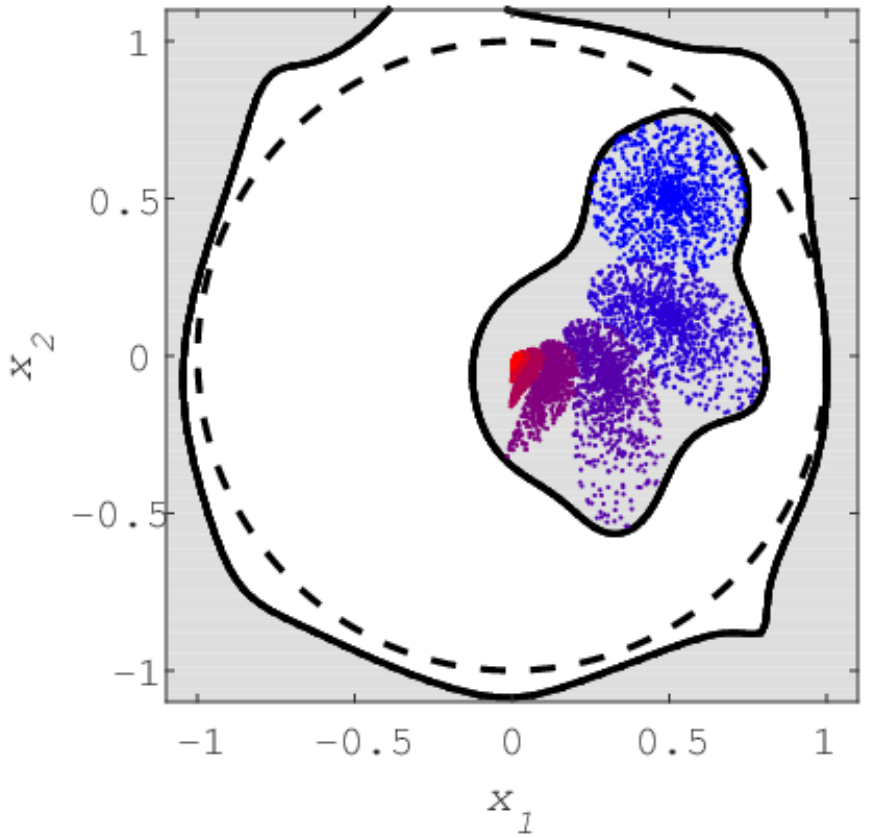}}
\subfigure[$2 r=12$]{
\includegraphics[scale=\sizesmallfig]{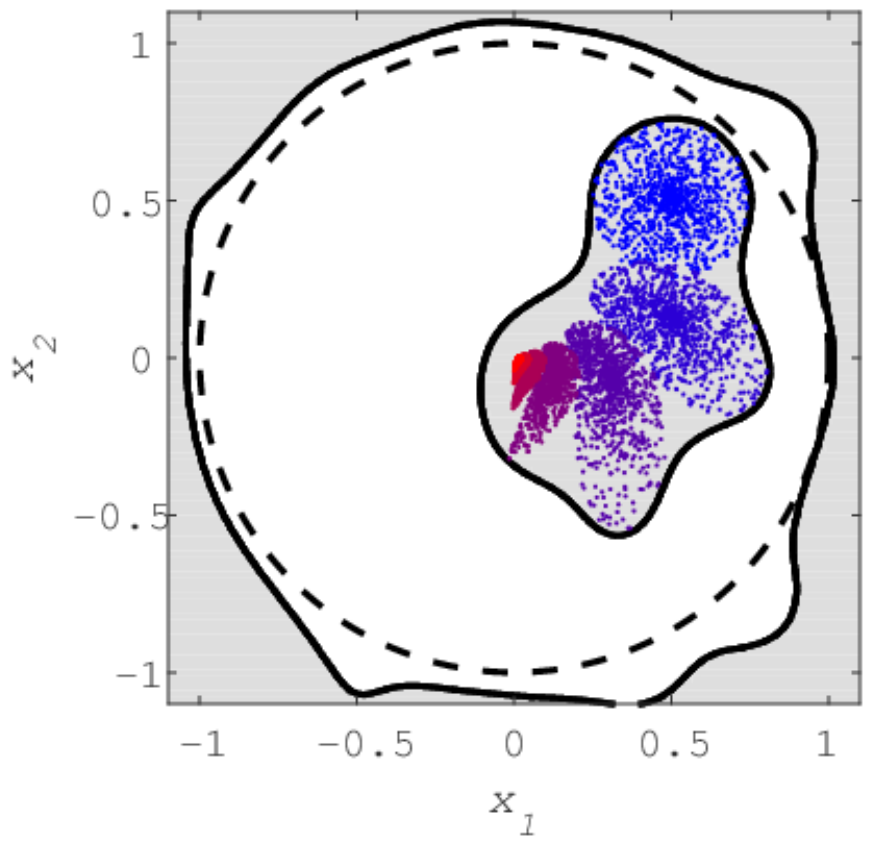}}
\subfigure[$2 r=14$]{
\includegraphics[scale=\sizesmallfig]{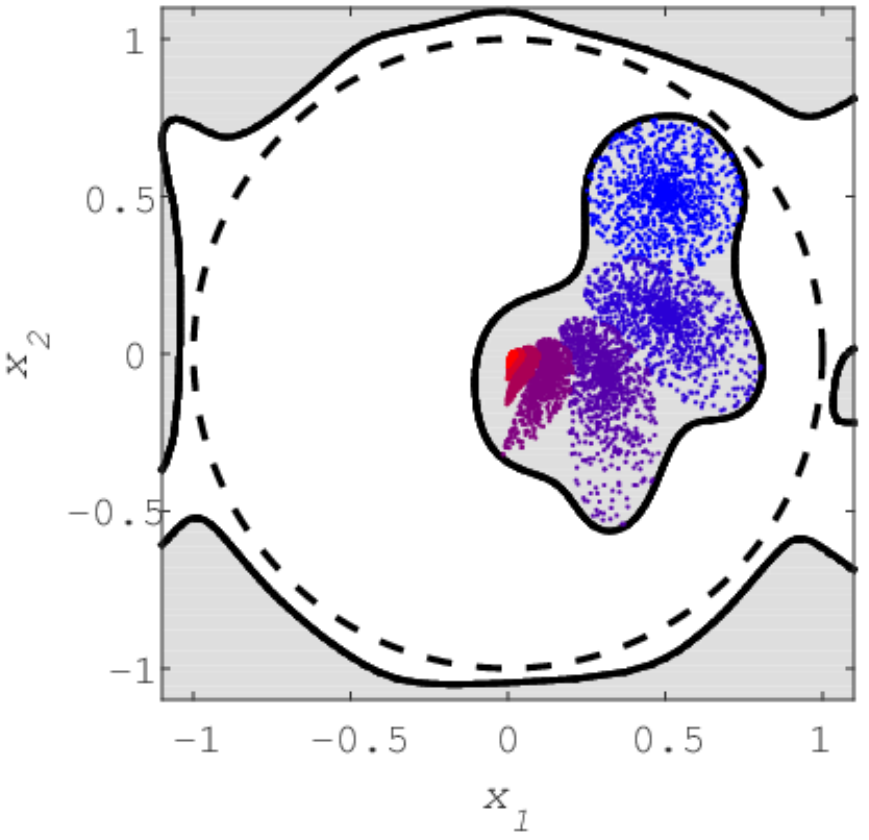}}
\caption{Outer approximations $\X^\infty_r$ (light gray) of $\X^\infty$ (color dot samples) 
for Example~\ref{ex:toy}, from $2 r=4$ to $2 r=14$.}
\label{fig:toy}
\end{figure}

\subsection{Cathala System}
\label{ex:cathala}
Consider the Cathala System (see~\cite[Section 7.1.2]{KHJ13mci}):
\begin{align*}
x_1^+ & := x_1 + x_2 \,, \\
x_2^+ & := -0.5952 + x_1^2 \,, 
\end{align*}
with initial state constraints $\X^0 := \{ \x \in \R^2 : (x_1 + 0.6)^2 + (x_2 - 0.5)^2  \leq 0.4^2 \}$ and state constraints $\X := \{ \x \in \R^2 : \| x \|_2^2  \leq 1.8^2 \}$.
The value $-0.5952$ corresponds to a parameter for which this system has an attractor (see Figure~\ref{fig:cathala}), the Cathala system being known to exhibit chaotic behavior~\cite{Mira}. 

\begin{figure}[!ht]
\centering
\subfigure[$2 r=6$]{
\includegraphics[scale=\sizesmallfig]{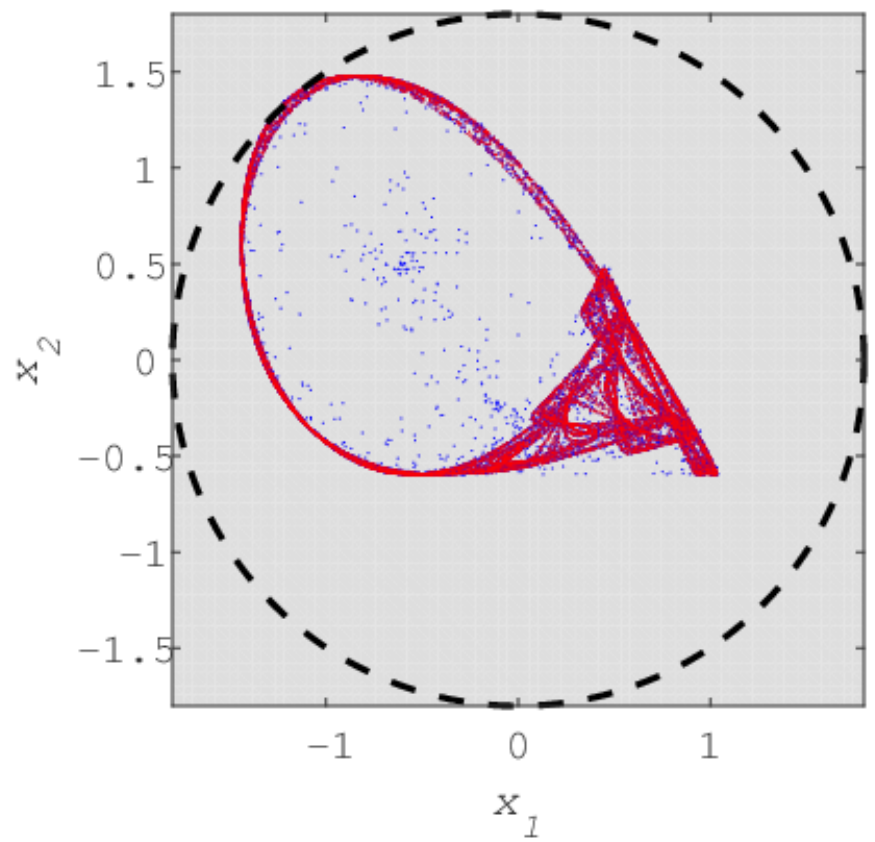}}
\subfigure[$2 r=8$]{
\includegraphics[scale=\sizesmallfig]{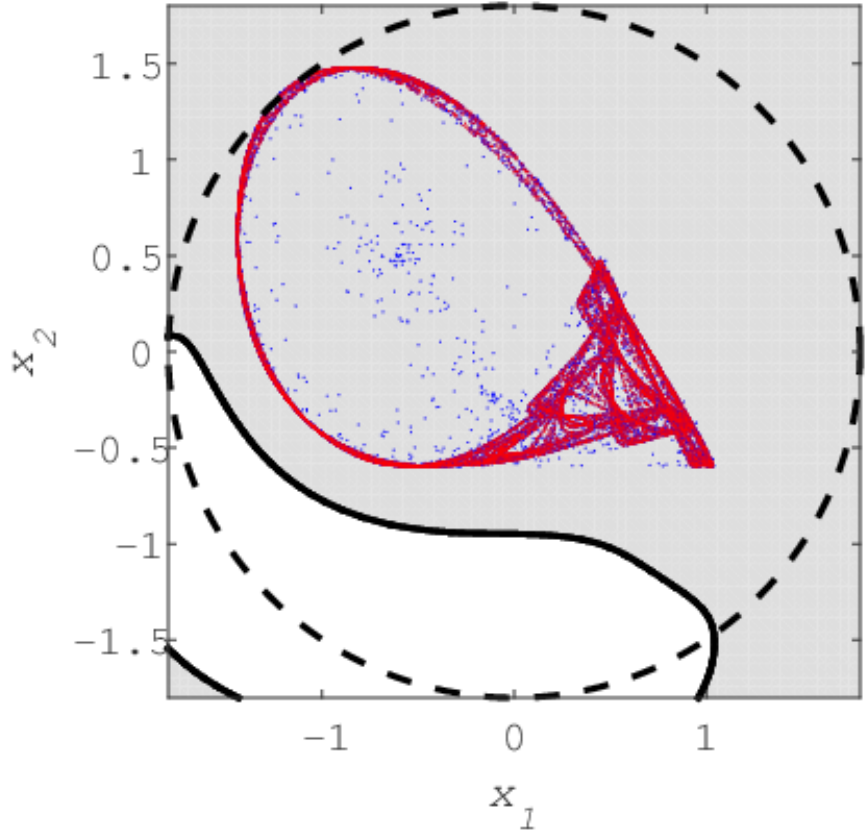}}
\subfigure[$2 r=10$]{
\includegraphics[scale=\sizesmallfig]{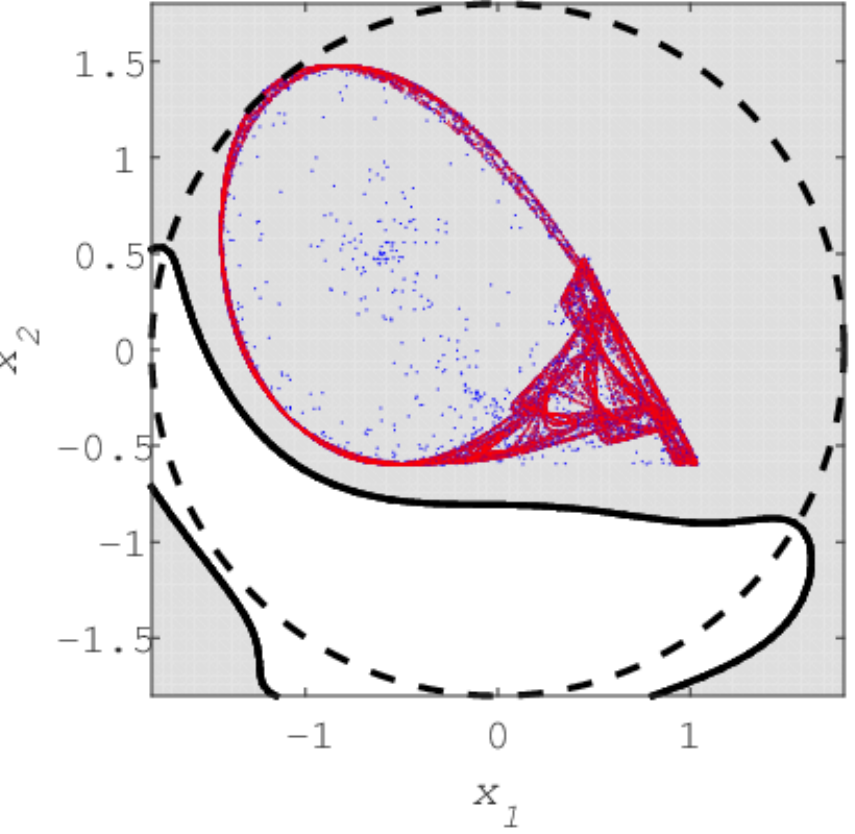}}\\
\subfigure[$2 r=12$]{
\includegraphics[scale=\sizesmallfig]{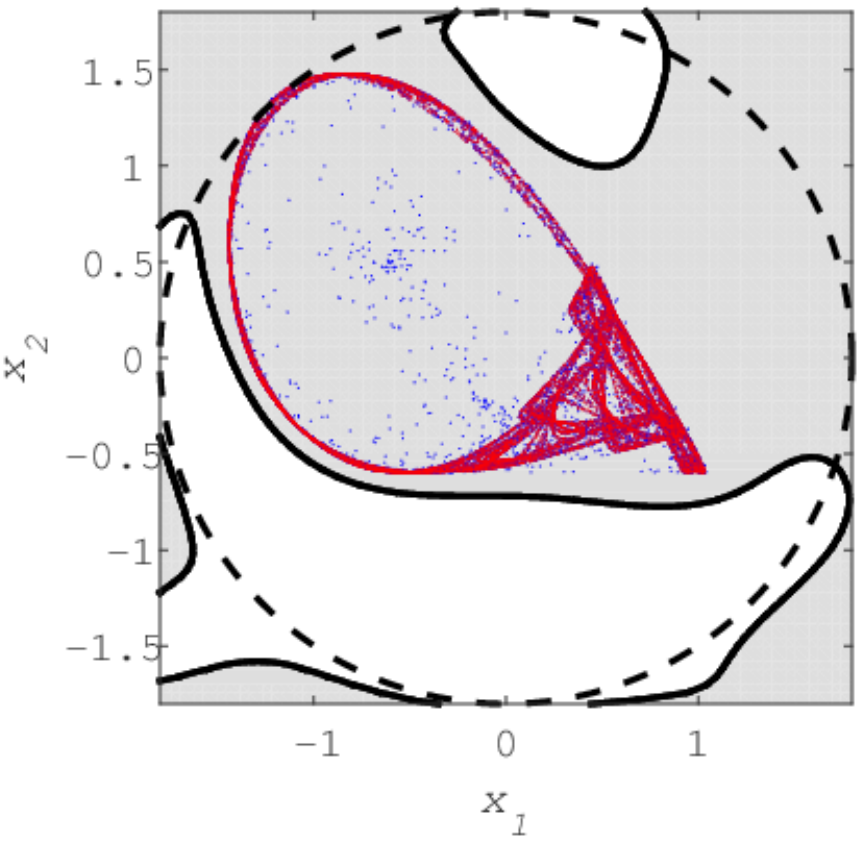}}
\subfigure[$2 r=14$]{
\includegraphics[scale=\sizesmallfig]{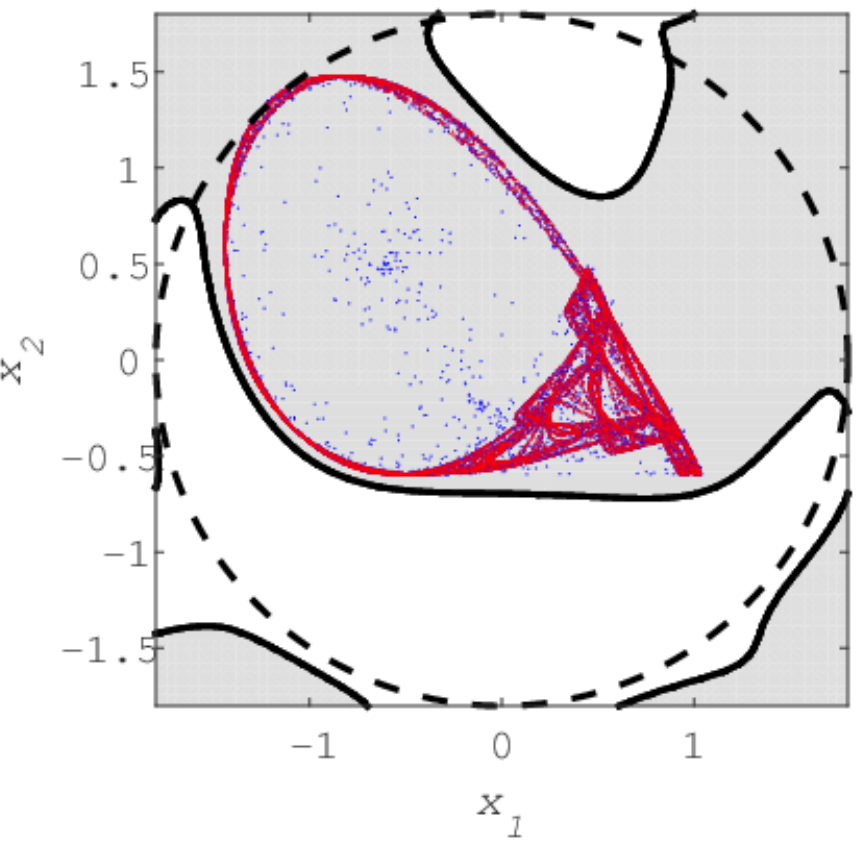}}
\subfigure[$2 r=16$]{
\includegraphics[scale=\sizesmallfig]{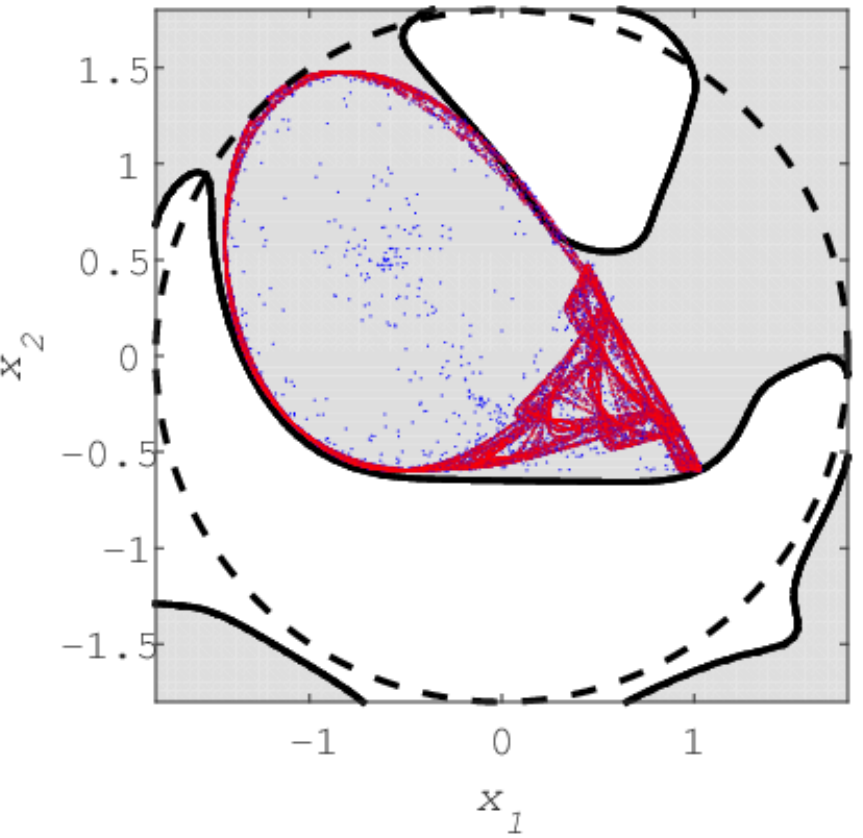}}
\caption{Outer approximations $\X^\infty_r$ (light gray) of $\X^\infty$ (color dot samples) 
for Example~\ref{ex:cathala}, from $2 r=6$ to $2 r=16$.}
\label{fig:cathala}
\end{figure}
\subsection{FitzHugh-Nagumo Neuron Model}
\label{ex:fitz}
Consider the discretized version (taken from~\cite[Section 5]{BenSassi2012}) of the FitzHugh-Nagumo model~\cite{FitzHugh61}, which is originally a continuous-time polynomial system modelling  the  electrical activity of a neuron:
\begin{align*}
x_1^+ & := x_1 + 0.2 (x_1 - x_1^3/3 - x_2 + 0.875)  \,, \\
x_2^+ & := x_2 + 0.2 (0.08 (x_1 + 0.7 - 0.8 x_2)) \,, 
\end{align*}
with initial state constraints $\X^0 := [1, 1.25] \times [2.25, 2.5]$ and state constraints $\X := \{ \x \in \R^2 : (\frac{x_1-0.1}{3.6})^2 + (\frac{x_2-1.25}{1.75})^2 \leq 1 \}$.
Figure~\ref{fig:fitz} illustrates that the outer approximations provide useful indications on the system behavior, in particular for higher values of $r$. Indeed $\X^\infty_{5}$ and $\X^\infty_{6}$ capture the presence of the central ``hole'' made by periodic trajectories and $\X^\infty_{7}$ shows that there is a gap between the first discrete-time steps and the iterations corresponding to these periodic trajectories. 
\begin{figure}[!ht]
\centering
\subfigure[$2 r=4$]{
\includegraphics[scale=\sizeellfig]{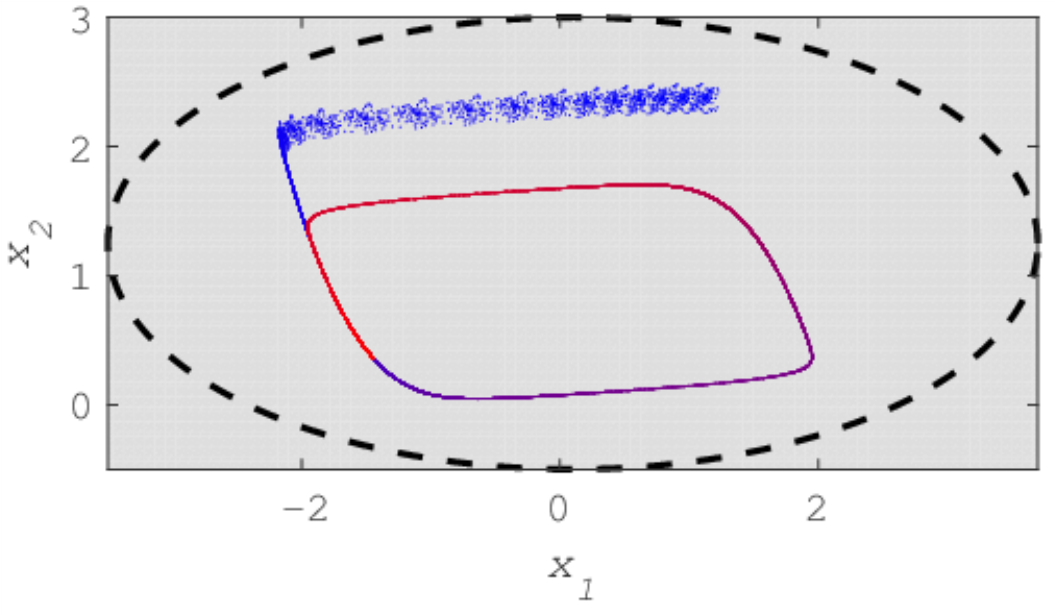}}
\subfigure[$2 r=6$]{
\includegraphics[scale=\sizeellfig]{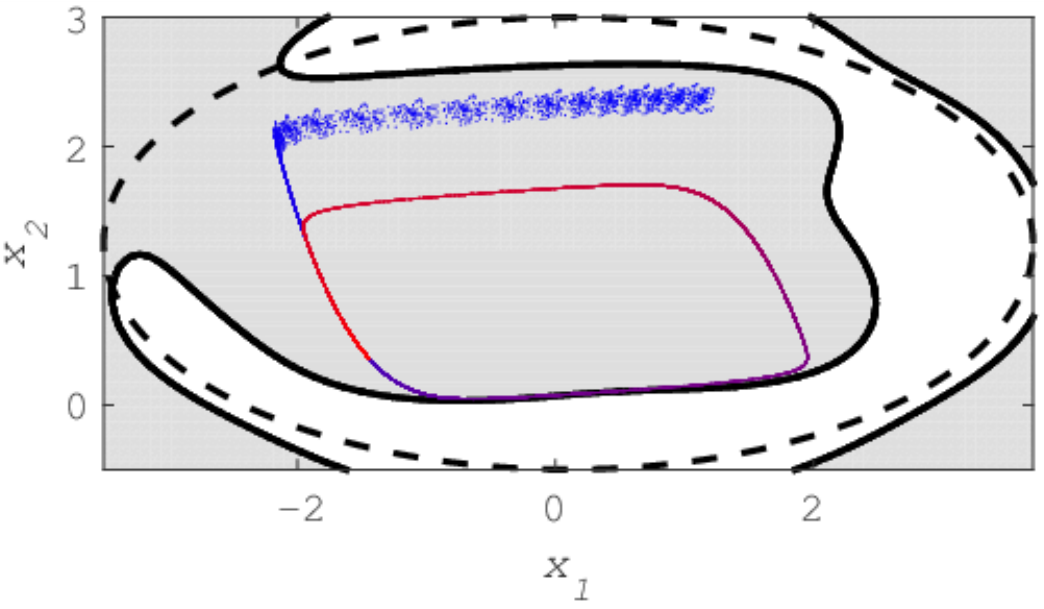}}\\
\subfigure[$2 r=8$]{
\includegraphics[scale=\sizeellfig]{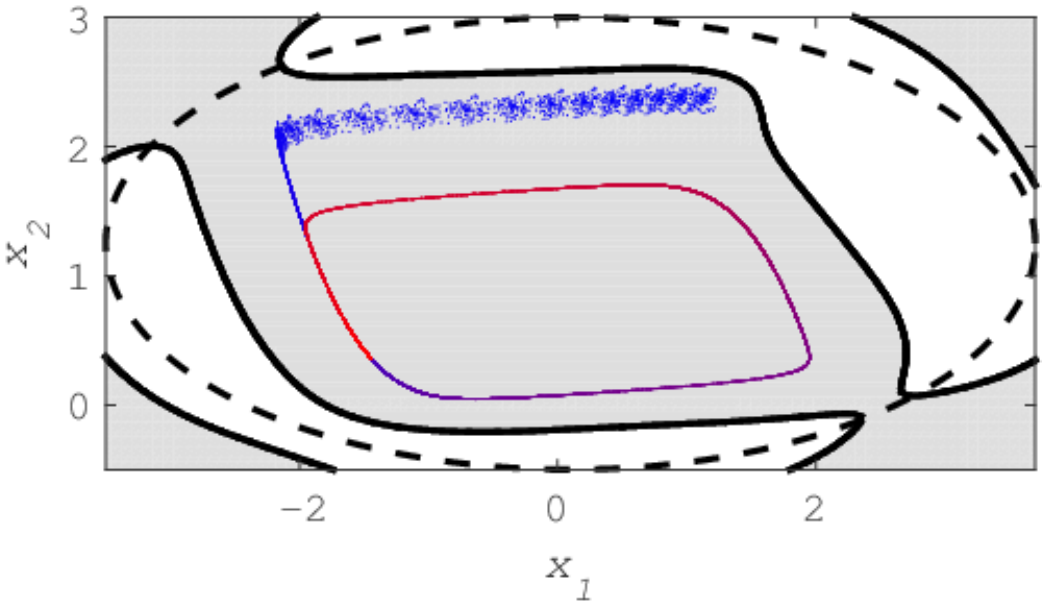}}
\subfigure[$2 r=10$]{
\includegraphics[scale=\sizeellfig]{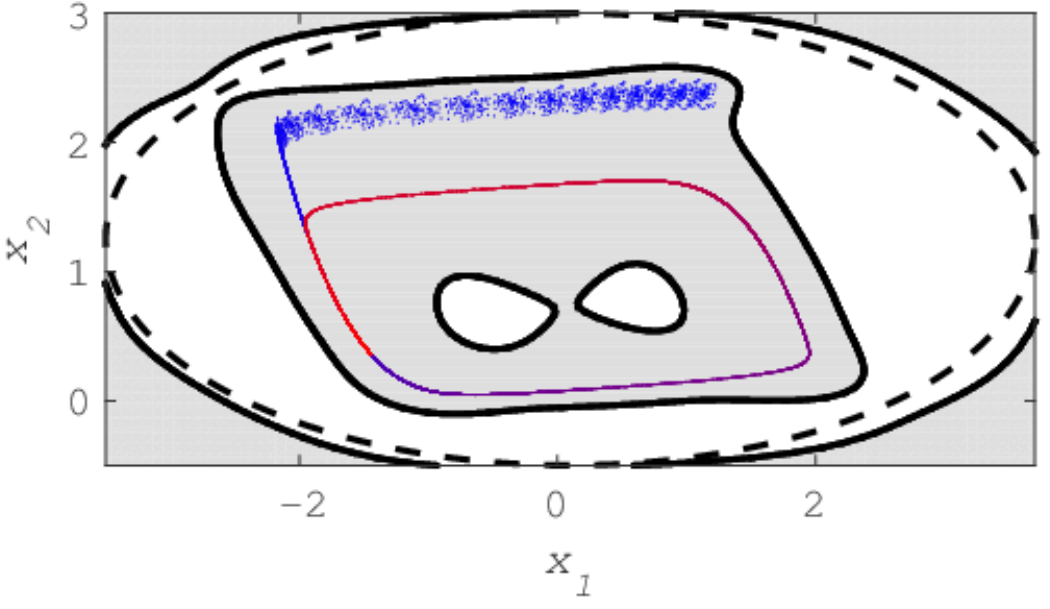}}\\
\subfigure[$2 r=12$]{
\includegraphics[scale=\sizeellfig]{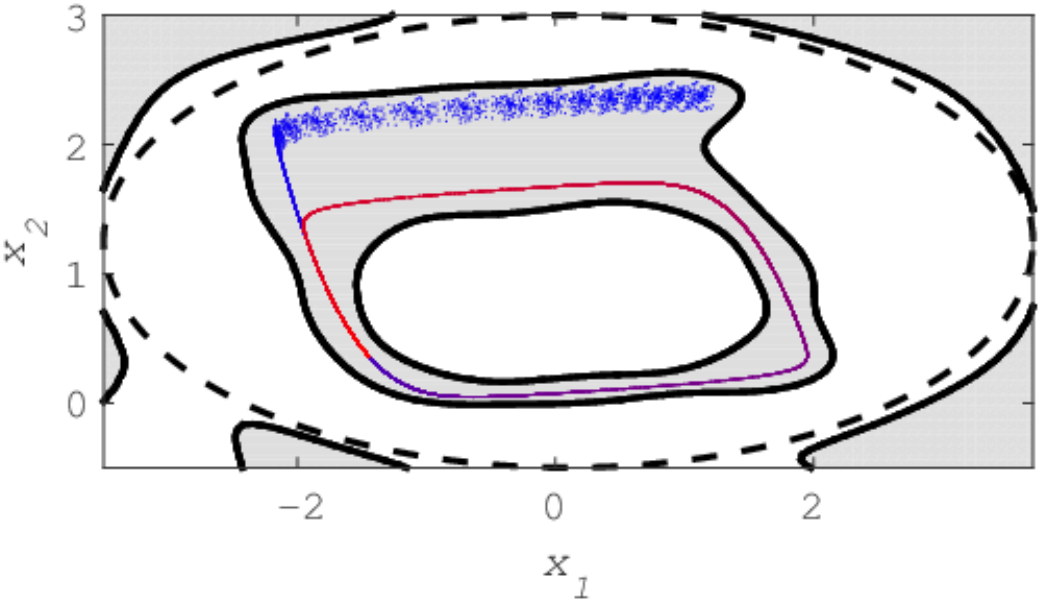}}
\subfigure[$2 r=14$]{
\includegraphics[scale=\sizeellfig]{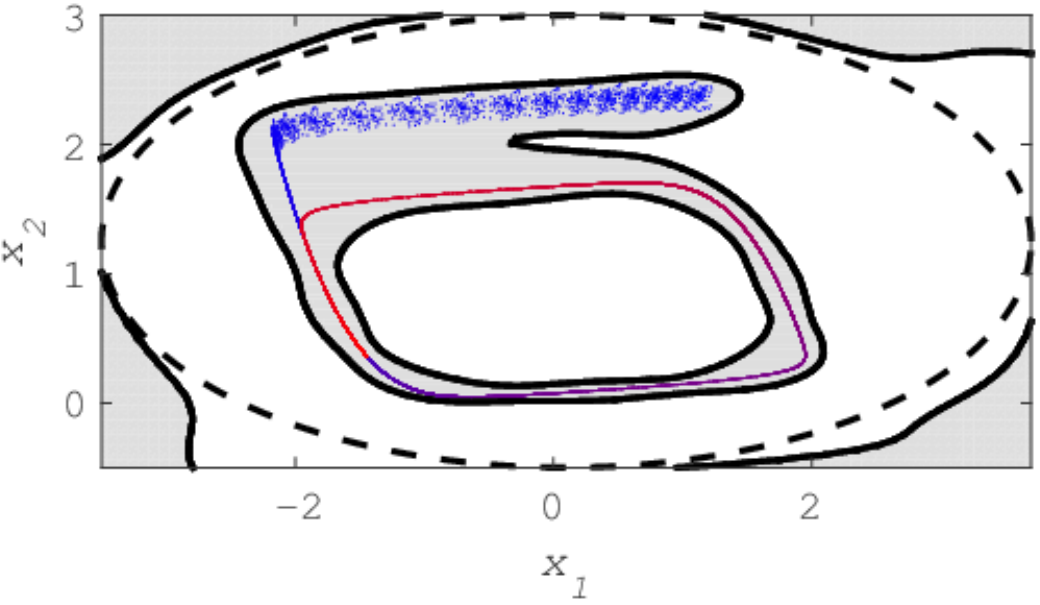}}
\caption{Outer approximations $\X^\infty_r$ (light gray) of $\X^\infty$ (color dot samples) 
for Example~\ref{ex:fitz}, from $2 r=4$ to $2 r=14$.}
\label{fig:fitz}
\end{figure}
\subsection{Julia Map}
\label{ex:julia}
Consider the discrete-time system $z^+ := z^2 + c$, with the state variable $z \in \C$ and parameter $c \in \C$. 
By setting $z = x_1 + i x_2$ and $c = c_1 + i c_2$, with $i$ the imaginary unit, we obtain the following equivalent quadratic two-dimensional formulation:
\begin{align*}
x_1^+ & := x_1^2 - x_2^2 + c_1 \,, \\
x_2^+ & := 2 x_1 x_2 + c_2 \,, 
\end{align*}
with initial state constraints $\X^0 := \{ \x \in \R^2 : \| \x \|_2^2 \leq 0.1^2 \}$ and state constraints $\X \subseteq [-1.2, 0.2] \times [-0.5, 0.6]$.
This recurrence allows to generate the filled Julia set, defined as the set of initial conditions for which the RS of the above quadratic system is bounded. Connectivity of the filled Julia set is ensured when $c$ belongs to the Mandelbrot set.
In~\cite[Section 7.1.3]{KHJ13mci}, the authors provide over approximations of the sets of initial condition for different values of the parameter $c$, in particular for the case $c=-0.7 + 0.2 i$ belonging to the Mandelbrot set and $c = -0.9 + 0.2 i$ which lies outside the Mandelbrot set. 

By contrast with the experimental results provided in~\cite[Section 7.1.3]{KHJ13mci}, Figure~\ref{fig:julia} depicts over approximations of the RS for different values of the parameter $c$. 
In particular, for cases where the parameter $c=-0.7 \pm 0.2 i$ lies inside the Mandelbrot set (corresponding to Figure~\ref{subfig:julia1} and Figure~\ref{subfig:julia2}), the over approximation $\X^\infty_{5}$ indicates that the trajectories possibly converge to an attractor. For cases where the parameter $c=-0.9 \pm 0.2 i$ lies outside the Mandelbrot set (corresponding to Figure~\ref{subfig:julia3} and Figure~\ref{subfig:julia4}), the over approximation $\X^\infty_{5}$ proves a disconnected behavior, possibly implying the presence of two attractors.
\begin{figure}[!ht]
\centering
\subfigure[$c_1 = -0.7$, $c_2 = 0.2$]{
\includegraphics[scale=\sizesmallfig]{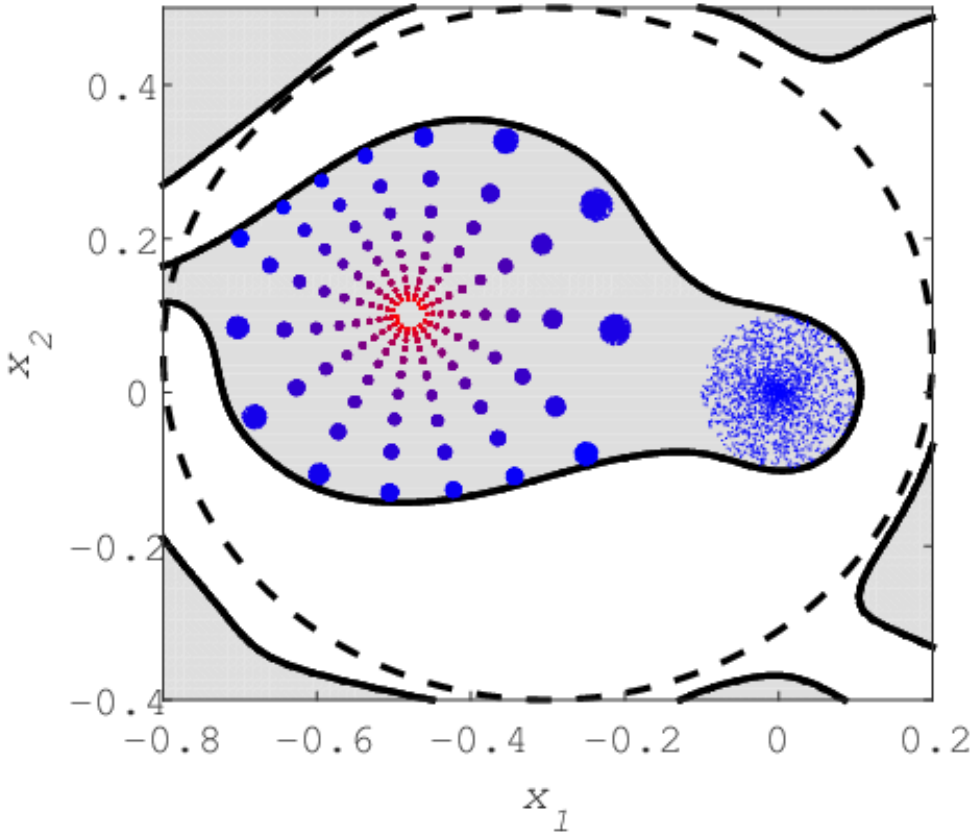}
\label{subfig:julia1}
}
\subfigure[$c_1 = -0.7$, $c_2 = -0.2$]{
\includegraphics[scale=\sizesmallfig]{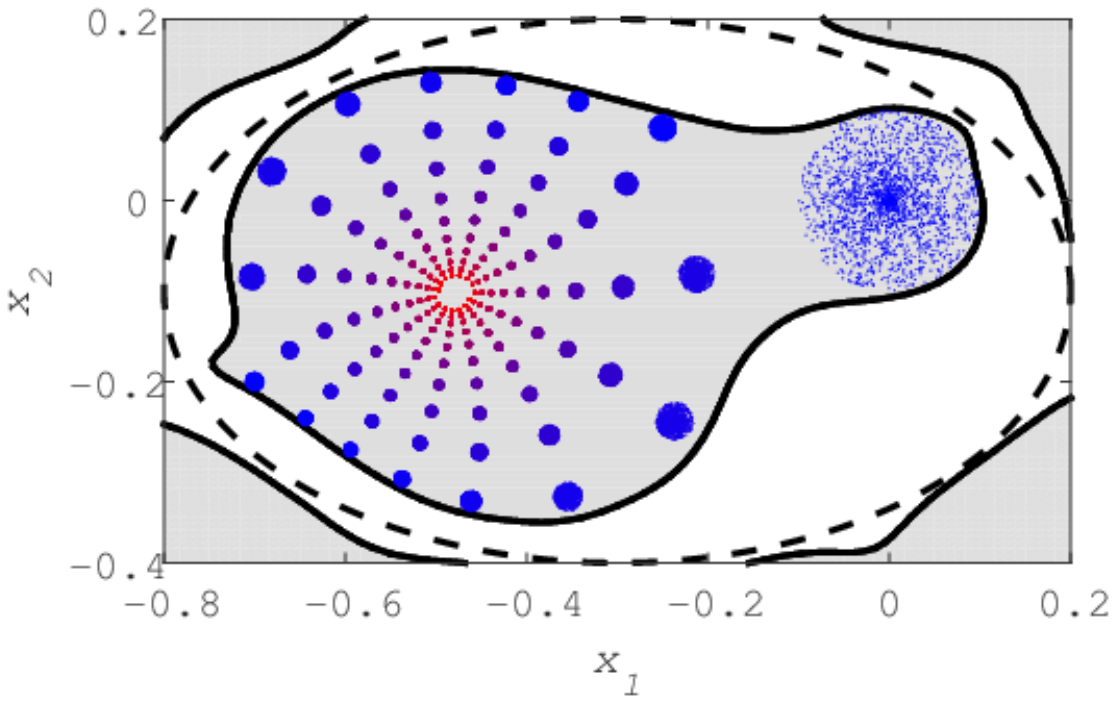}
\label{subfig:julia2}
}\\
\subfigure[$c_1 = -0.9$, $c_2 = 0.2$]{
\includegraphics[scale=\sizesmallfig]{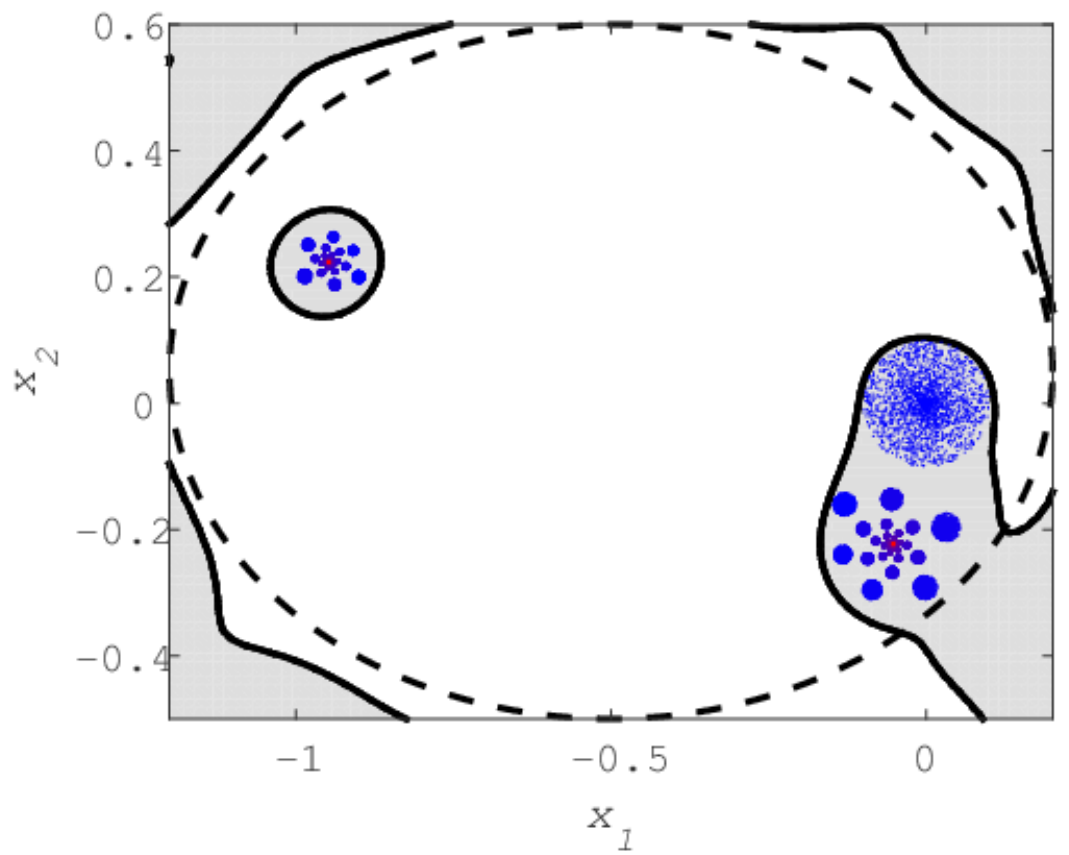}
\label{subfig:julia3}
}
\subfigure[$c_1 = -0.9$, $c_2 = -0.2$]{
\includegraphics[scale=\sizesmallfig]{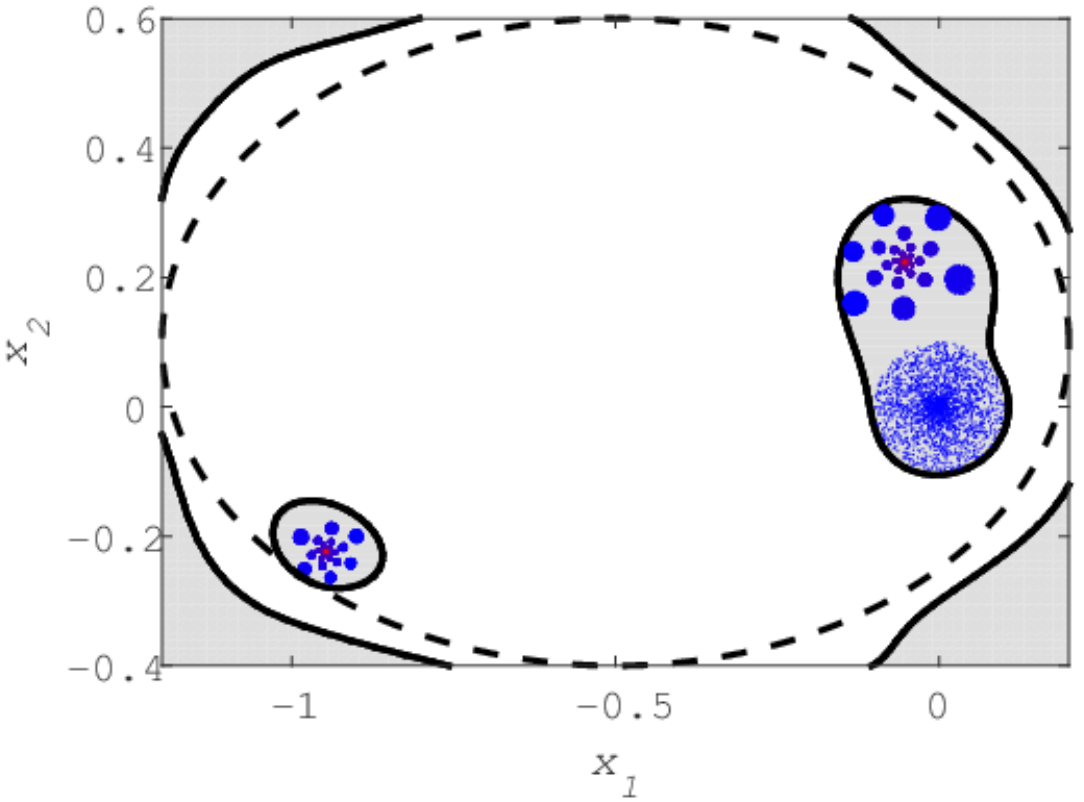}
\label{subfig:julia4}
}\\
\caption{Outer approximations $\X^\infty_r$ (light gray) of $\X^\infty$ (color dot samples) 
for Example~\ref{ex:julia}, for $2 r = 10$ and different values of the Julia parameter $c$.}
\label{fig:julia}
\end{figure}

\subsection{Phytoplankton Growth Model}
\label{ex:phytoplankton}
Consider the discretized version of the Phytoplankton growth model (also taken from~\cite[Section 5]{BenSassi2012}). This model is obtained after making assumptions, corroborated experimentally by biologists in order to represent such growth phenomena~\cite{Bernard02}, yielding the following discrete-time polynomial system:
\begin{align*}
x_1^+ & := x_1 + 0.01 (1 - x_1 - 0.25 x_1 x_2)   \,, \\
x_2^+ & := x_2 + 0.01 (2 x_3 - 1) x_2 \,, \\
x_3^+ & := x_3 + 0.01 (0.25 x_1 - 2 x_3^2) \,,
\end{align*}
with initial state constraints $\X^0 := [-0.3, -0.2]^2 \times [-0.05, 0.05]$ and state constraints $\X := [-0.5, 1.5] \times [-0.5, 0.5]^2$.
Figure~\ref{fig:phytoplankton} illustrates the system convergence behavior towards an equilibrium point for initial conditions near the origin. \new{One way to obtain more accurate information on such systems would be to design a subdivision procedure (e.g., with branch-and-bound techniques), which boils down to zooming on  specific areas of the RS. }
\begin{figure}[!ht]
\centering
\subfigure[$2 r=4$]{
\includegraphics[scale=\sizephytofig]{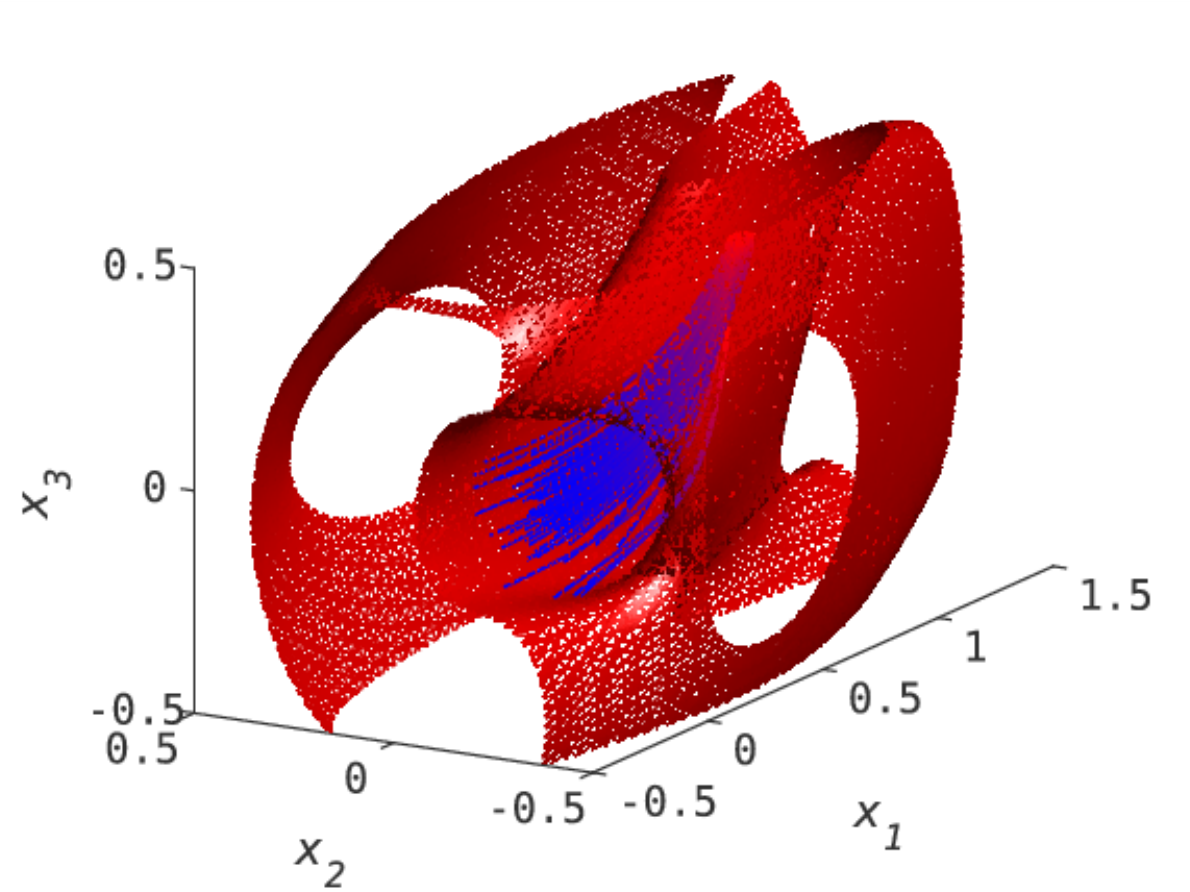}}
\subfigure[$2 r=6$]{
\includegraphics[scale=\sizephytofig]{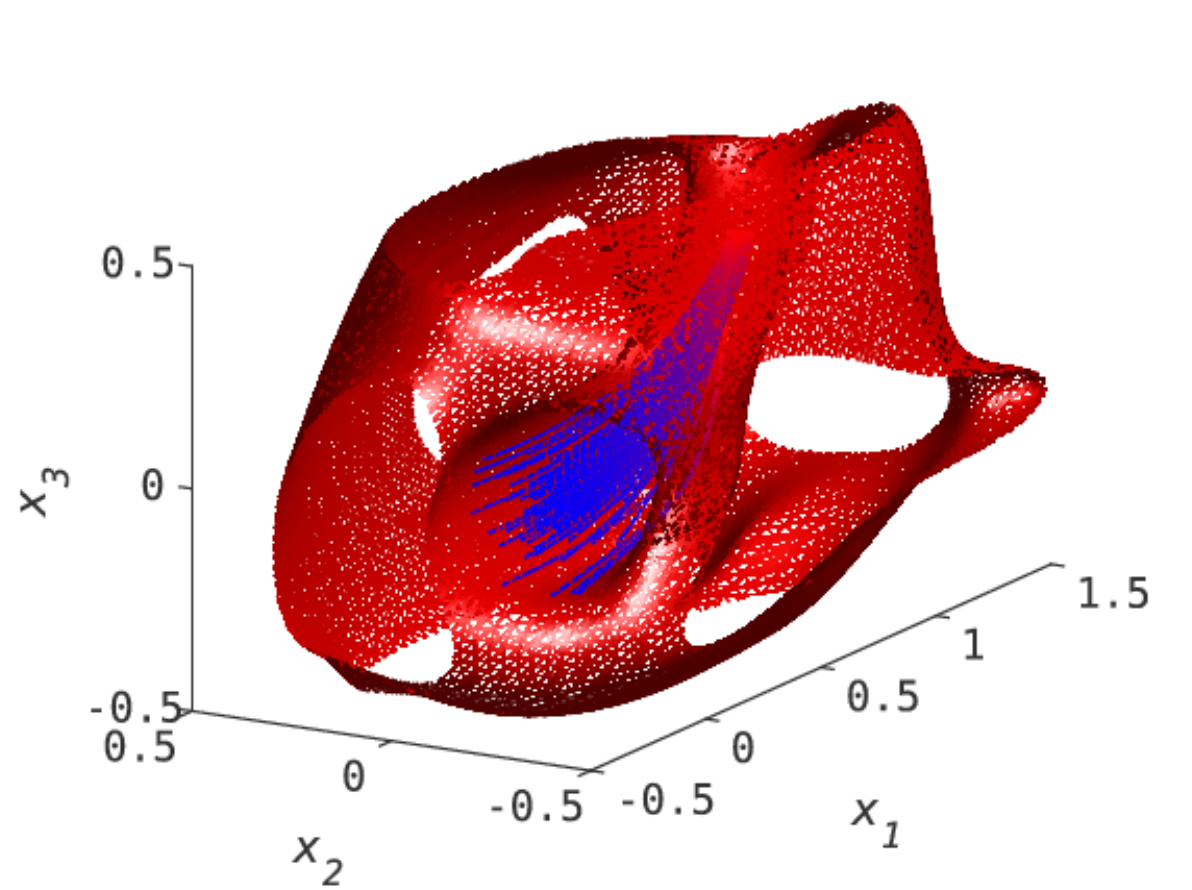}}
\subfigure[$2 r=8$]{
\includegraphics[scale=\sizephytofig]{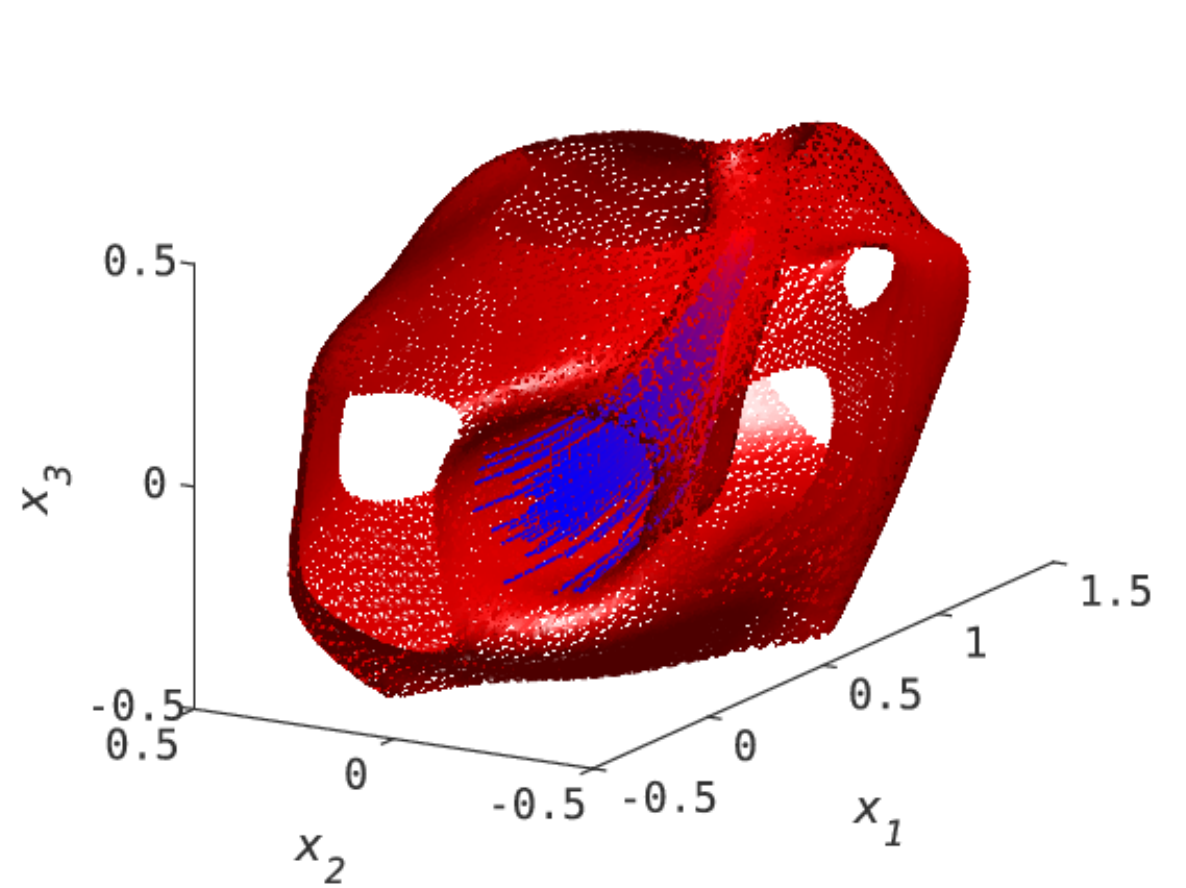}}\\
\subfigure[$2 r=10$]{
\includegraphics[scale=\sizephytofig]{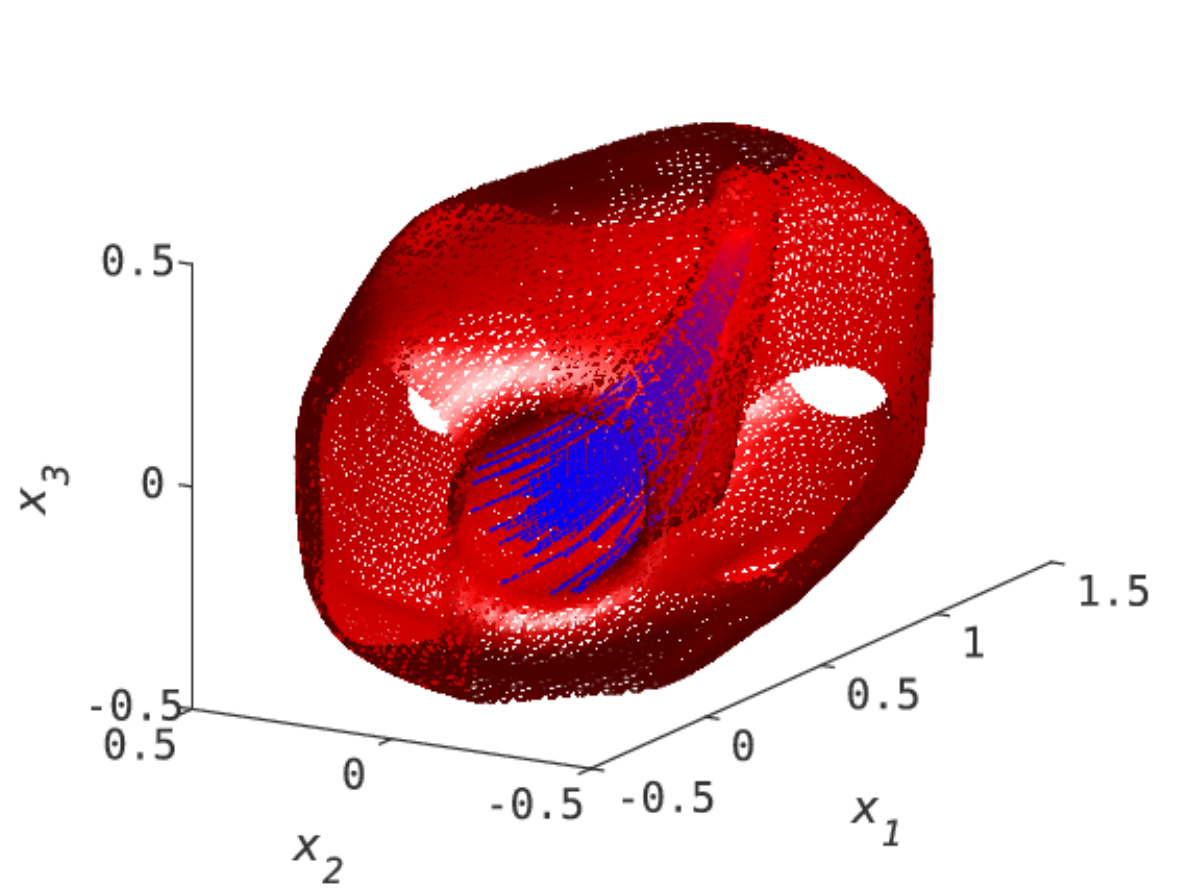}}
\subfigure[$2 r=12$]{
\includegraphics[scale=\sizephytofig]{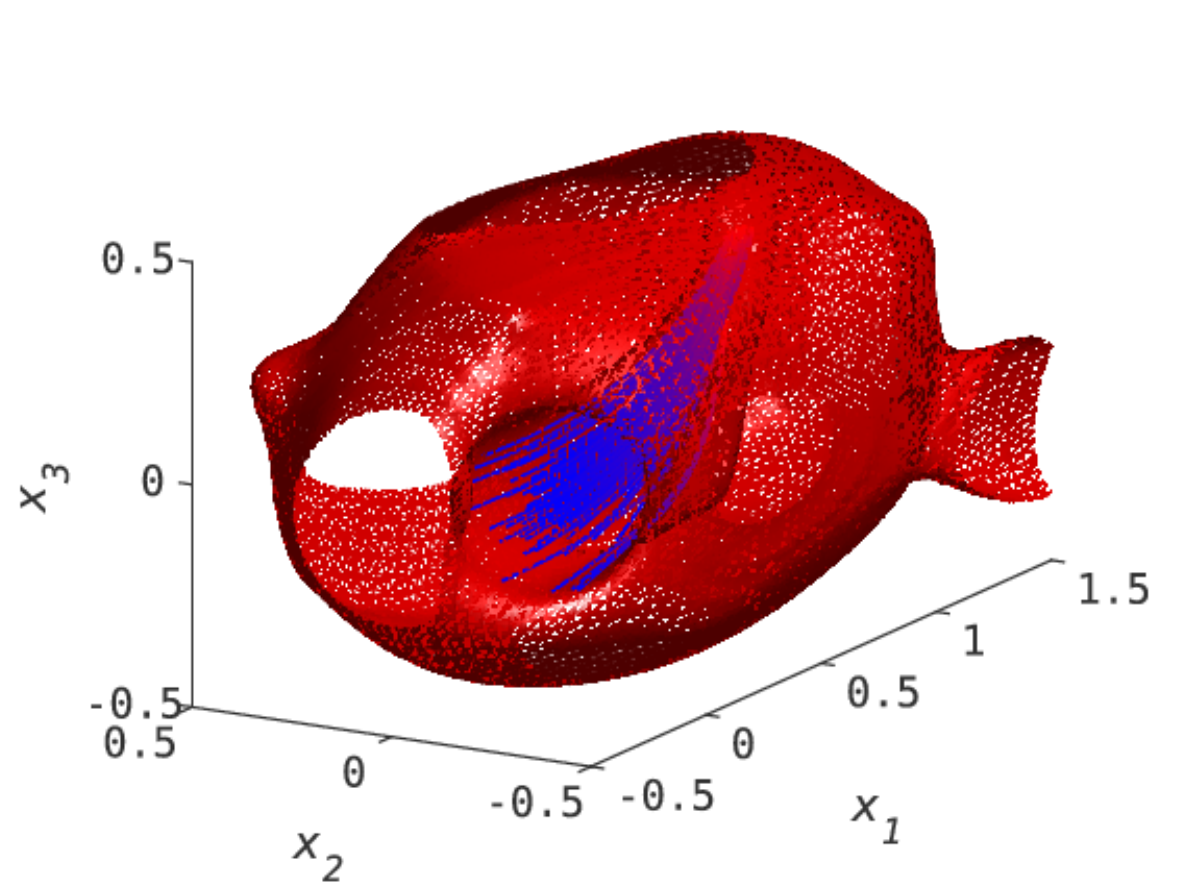}}
\subfigure[$2 r=14$]{
\includegraphics[scale=\sizephytofig]{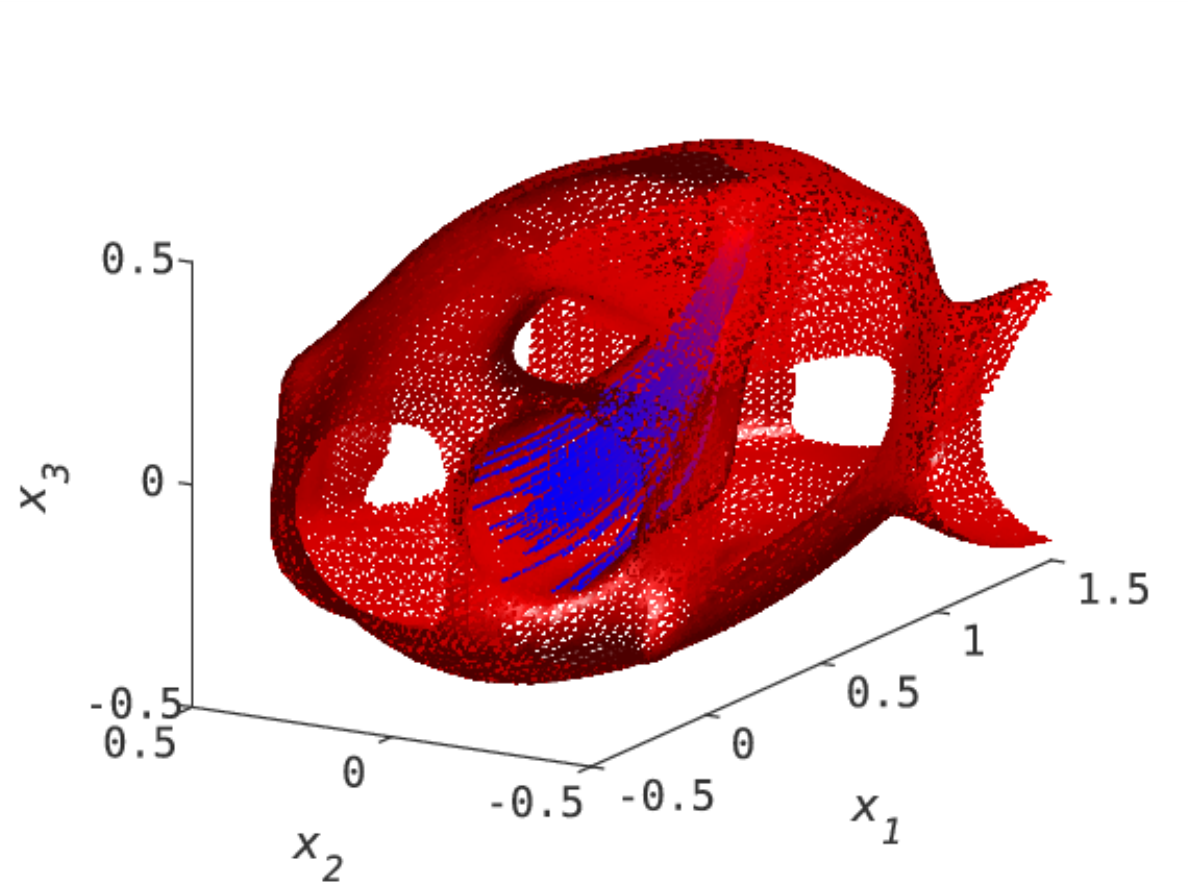}}
\caption{Outer approximations $\X^\infty_r$ (red) of $\X^\infty$ (color dot samples) 
for Example~\ref{ex:phytoplankton}, from $2 r=4$ to $2 r=14$.}
\label{fig:phytoplankton}
\end{figure}
\section{Conclusion and Perspectives}
\label{sec:end}
This paper presented an infinite-dimensional primal-dual LP characterization of the (forward) reachable set (RS) for discrete-time polynomial systems with semialgebraic initial and general state constraints. The problem can be practically handled through solving a hierarchy of finite dimensional primal-dual SDP relaxations. 

In particular, the hierarchy of dual SDP problems yields sequences of polynomials of increasing degrees, allowing to construct certified outer approximations of the RS while ensuring convergence guarantees (w.r.t.~the $L_1$ norm) to the indicator function of the RS when the mass of some occupation measure is bounded.
Our approach happens to be not only theoretically consistent but also practically efficient.

In some cases, it is possible to complement the hierarchy of convergent outer approximations of set of interest (ROA, MCI) by providing a sequence of inner approximations. 
For instance, the work~\cite{KHJ12innerROA} uses similar tools from measure theory to derive such inner approximations of the ROA. 
Future research perspectives include the study a complementary hierarchy of inner approximations for the RS, in the spirit of~\cite{KHJ12innerROA}. We also intend to investigate the RS problem for continuous time polynomial systems as for the infinite-dimensional convex modeling of the maximum controlled invariant, with infinite horizon~\cite{KHJ13mci}. In addition, it would be worth to apply the framework of~\cite{HK14roa}, relying on occupation measures to approximate the region of attraction (ROA). A time-reversal argument would allow to formulate the RS problem in continuous time and finite horizon as ROA characterization.
Finally, we also intend to develop a formally certified framework, inspired from~\cite{jfr14}, in order to guarantee the correctness of the over approximations.

\bibliographystyle{plain}


\end{document}